\documentclass{article}
   \usepackage{amsthm}
   \usepackage{amsmath}
   \usepackage{amssymb}
   \usepackage{amsfonts}
   \usepackage{amscd}
   \usepackage{mathrsfs}  
   \usepackage{bm}
   
  \newtheorem{theorem}{Theorem}[section]
  \newtheorem{definition}[theorem]{Definition}
  \newtheorem{remark}[theorem]{Remark} 
  \newtheorem{lemma}[theorem]{Lemma}
  \newtheorem{proposition}[theorem]{Proposition}
  \newtheorem{corollary}[theorem]{Corollary}
  \newtheorem{claim}[theorem]{Claim}
  \newtheorem{example}[theorem]{Example}
  
   \makeatletter
\@addtoreset{equation}{section}

\makeatother 
    
\begin{document}

         \title{ Mixed trTLEP-structures and mixed Frobenius structures}
            \author{Yota Shamoto}
            \date{\today}
            
          \maketitle

        \begin{abstract}
        We introduce the notion of mixed trTLEP-structures 
        and prove that a mixed trTLEP-structure with some conditions
        naturally induces a mixed Frobenius manifold.
        This is a generalization of the reconstruction theorem of Hertling and Manin. 
        As a special case, we also show that 
        a graded polarizable variation of mixed Hodge structure
        with $H^2$-generation condition gives rise to a family of mixed Frobenius manifolds.
        It implies that there exist mixed Frobenius manifolds associated to local B-models.
        
%
        \end{abstract}
        
        \tableofcontents
        \section{Introduction}
         
           \subsection{Frobenius manifolds and their comparison}          
           A Frobenius manifold is a complex manifold 
           whose tangent bundle is equipped with a commutative product, a metric and a global section 
           satisfying a kind of integrability condition.         
           The notion of Frobenius manifold was introduced by Dubrovin \cite{Dub1} 
           for the investigation of 2D topological field theories.
           (The equivalent structure, called flat structure, was introduced by K. Saito \cite{K.Saito1}.)
           It has been shown that Frobenius manifolds naturally arise in many theories: 
           the invariant theory of Weyl groups \cite{K.T};
           singularity theory \cite{Sab2, Sab3, Sab4, K.T};  
           Gromov-Witten theory \cite{Kontz Man}; 
           the deformation theory
           of $A_\infty$-structures \cite{Konts Bara}, 
           etc.
                               
           In some cases, we have an interesting isomorphism of Frobenius manifolds in two different theories.
           For instance, one of the goals in mirror symmetry 
           is to prove that a Frobenius manifold constructed in a A-model is isomorphic 
           to the one constructed in the corresponding B-model.

           However, in general, it is difficult to compare Frobenius manifolds 
           $\mathscr{F}^{(0)}$ and $\mathscr{F}^{(1)}$. 
           A useful strategy is to split the problem into two steps as follows.
             \begin{description}
             \item[Step 1.] 
              Show that each $\mathscr{F}^{(i)}$ is constructed from more restricted set of data 
              $\mathcal{T}^{(i)}$ $(i=0,1)$.
             \item[Step 2.]
              Show that $\mathcal{T}^{(0)}\simeq \mathcal{T}^{(1)}$.
           \end{description}         
           The fact in Step 1 is called (re-)construction theorem. 
           It depends on the problem which data we choose.
           Following \cite{H.M1}, let us consider the case of trTLEP-structure.
           
           Let $M$ be a complex manifold and 
           $j_\lambda:\mathbb{P}_\lambda^1\times M\to \mathbb{P}_\lambda^1\times M$ a map defined by 
           $j_\lambda(\lambda,t):=(-\lambda,t)$ where $\lambda$ 
           is non-homogeneous coordinate on $\mathbb{P}^1_\lambda$
           and $t$ is a point in $M$.
           For an integer $k$, 
           trTLEP($k$)-structure on $M$ is a tuple $(\mathcal{H},{\nabla},P)$ with following properties
           (Definition \ref{def of trTLEP(k)}).
           $\mathcal{H}$ is a holomorphic vector bundle on $\mathbb{P}_\lambda^1\times M$ 
           trivial along $\mathbb{P}^1_\lambda$, 
           ${\nabla}$ is a meromorphic flat connection on $\mathcal{H}$;
            \begin{equation}\label{intro-conn}
              {\nabla}\colon \mathcal{H}\to 
              \mathcal{H}\otimes
              \Omega^1_{\mathbb{P}^1_\lambda\times M}
              \big{(}\log (\{0,\infty\}\times M)\big{)}
              \otimes\mathcal{O}_{\mathbb{P}_\lambda^1\times M}(\{0\}\times M)
              .
            \end{equation}
           $P$ is ${\nabla}$-flat non-degenerate $(-1)^k$-symmetric pairing 
           \begin{equation}
           P:\mathcal{H}\otimes j_\lambda^*\mathcal{H}\to 
             \mathcal{O}_{\mathbb{P}^1_\lambda\times M}(-k\{0\}\times M + k\{\infty\}\times M).
           \end{equation}

           Hertling and Manin \cite{H.M1} showed the construction theorem for trTLEP-structure. 
           In other words, they proved that 
           a trTLEP-structure with some condition uniquely induces a Frobenius manifold.
           Then, they applied the construction theorem to compare the Frobenius manifolds 
           constructed from 
           isolated singularities,
           Frobenius manifolds associated to 
           variation of polarized Hodge structure of some family of hypersurfaces,  
           and super Frobenius manifolds in the deformation theory of $A_\infty$-structures. 
           
           Reichelt \cite{R1} defined the notion of logarithmic trTLEP-structure 
           as a generalization of trTLEP-structure
           and proved the construction theorem for logarithmic Frobenius manifolds.
           Reichelt and Sevenheck \cite{RS} used the construction theorem to 
           refine the mirror symmetry theorem \cite{Giv} for 
           weak Fano toric manifolds. 
           Here, we note that in \cite{RS}, 
           the result of Givental \cite{Giv} 
           plays an important role at Step 2 of the 
           strategy above.

           \subsection{Construction theorem for mixed Frobenius manifolds}
           The first main theorem of this paper is the construction theorem for mixed Frobenius manifolds.
           We introduce the notion of mixed trTLEP-structure and 
           show that a mixed trTLEP-structure with some condition 
           naturally induces a mixed Frobenius manifold.
           The notion of mixed Frobenius manifolds was introduced by 
           Konishi and Minabe \cite{konisi1} to understand the local mirror symmetry.
           Here we shall explain the notions of mixed trTLEP-structures and mixed Frobenius structures, 
           and the statement of the construction theorem. 
           The applications of the construction theorem 
           to the local mirror symmetry will be discussed in \S\ref{intro local B}
           and \S\ref{intro local A}.
             \subsubsection{Mixed Frobenius structures and mixed trTLEP-structures}\label{intro Frobenius}
             Let $M$ be a complex manifold equipped with holomorphic vector fields $e$ and $E$.
             Suppose that the tangent bundle $\Theta_M$ has 
             an associative commutative product $\circ$, a torsion free flat connection ${\bm \nabla}$, 
             and a ${\bm \nabla}$-flat increasing
             filtration $\mathcal{I}=(\mathcal{I}_k\mid k\in \mathbb{Z})$.
             If the tuple $(\circ,{\bm \nabla},e,E,\mathcal{I})$ and sequence of metrics $g=(g_k\mid k\in\mathbb{Z})$ 
             on $\mathrm{Gr}^\mathcal{I}\Theta_M=\bigoplus_{k\in \mathbb{Z}}\mathrm{Gr}^\mathcal{I}_k\Theta_M$   
             satisfies some conditions, we call the tuple $\mathscr{F}:=(\circ,{\bm \nabla},e,E,\mathcal{I},g)$
             a mixed Frobenius structure (MFS)
             on $M$. A complex manifold equipped with a MFS is called mixed Frobenius manifold
             (Definition \ref{def mixed Frob}).
                          
             Similarly, a mixed trTLEP-structure is a trTLEP-structure with a filtration.
             Let $\mathcal{H}$ be a holomorphic vector bundle on $\mathbb{P}^1_\lambda\times M$ 
             trivial along $\mathbb{P}^1_\lambda$
             and ${\nabla}$ a meromorphic flat connection on $\mathcal{H}$ as in (\ref{intro-conn}).
             If we are given an increasing filtration ${W}=({W}_k\mid k\in\mathbb{Z})$ 
             of ${\nabla}$-flat subbundle 
             on $\mathcal{H}$ and pairings $P=(P_k\mid k\in\mathbb{Z})$ on $\mathrm{Gr}^{{W}}\mathcal{H}$
             such that $(\mathrm{Gr}^{W}_k(\mathcal{H}),{\nabla},P_k)$ 
             is a trTLEP($-k$)-structure
             for any $k$, 
             then we call the tuple $\mathcal{T}:=(\mathcal{H},{\nabla},{W},P)$ 
             mixed trTLEP-structure on $M$ (Definition \ref{def mixed trTLEP}). 
             As we will see in Proposition
             \ref{proposition mixed Frobenius and mixed trTLEP}, 
             a mixed Frobenius structure $\mathscr{F}$ on $M$ naturally induces 
             a mixed trTLEP-structure $\mathcal{T}(\mathscr{F})$ on $M$.
                        
             Let $f:M_0\to M_1$ be a holomorphic map between complex manifolds. 
             A mixed trTLEP-structure $\mathcal{T}$ on $M_1$ 
             naturally induces mixed trTLEP-structure $f^*\mathcal{T}$ on $M_0$.
             In particular, if we are given MFS $\mathscr{F}$ on a complex manifold $M$
             and closed embedding $\iota: M\hookrightarrow \widetilde{M}$ 
             then we have the mixed trTLEP-structure $\mathcal{T}=\iota^*\mathcal{T}(\mathscr{F})$.
             This plays the role of the ^^ ^^ restricted set of data'' in Step 1 in the strategy.

             \subsubsection{Unfolding of mixed trTLEP-structure and the construction theorem}\label{intro unfolding}
             Let $(M,0)$ be a germ of a complex manifold and $\mathcal{T}$ be a mixed trTLEP-structure on $(M,0)$.
             An unfolding of $\mathcal{T}$ is a tuple 
             $\big{(}(\widetilde{M},0),\widetilde{\mathcal{T}},\iota, i\big{)}$
             where 
             $\iota:(M,0)\hookrightarrow (\widetilde{M},0)$ is a closed embedding , 
             $\widetilde{\mathcal{T}}$ is a mixed trTLEP-structure on $(\widetilde{M},0)$,
             and  
             $i:\mathcal{T}\xrightarrow{\sim} \iota^*\widetilde{\mathcal{T}}$
             is an isomorphism of mixed trTLEP-structures.
             We can define the notion of morphisms of unfoldings of 
             $\mathcal{T}$.
             Hence we get the category of unfoldings of $\mathcal{T}$
             denoted by $\mathfrak{Unf}_\mathcal{T}$
             (Definition \ref{def of unfolding}).   
             
             If there exists a terminal object in $\mathfrak{Unf}_\mathcal{T}$, 
             we call it universal unfolding of $\mathcal{T}$. 
             We show that there is a universal unfolding 
             $\big{(}(\widetilde{M},0),\widetilde{\mathcal{T}},\iota, i\big{)}$
             of $\mathcal{T}$
             under some conditions (Theorem \ref{unfolding theorem for mixed trTLEP-structure}).
             Moreover, we show that $\widetilde{\mathcal{T}}$ is isomorphic to 
             $\mathcal{T}(\mathscr{F})$ for a MFS $\mathscr{F}$ on $(\widetilde{M},0)$
             (Corollary \ref{construction theorem of mixed Frobenius manifold}). 
             Hence we get the first main theorem of this paper as follows.
     
             \begin{theorem}[Theorem \ref{unfolding theorem for mixed trTLEP-structure},
             Corollary \ref{construction theorem of mixed Frobenius manifold}]
             \label{intro main theorem}
             Let $\mathcal{T}$ be a mixed {\rm trTLEP}-structure on a germ of complex manifold $(M,0)$.
             Assume that $\mathcal{T}$ satisfies ^^ ^^ some conditions''. 
             Then there exists $($uniquely up to isomorphisms$)$
             a mixed Frobenius structure $\mathscr{F}$
             on a germ of a complex manifold $(\widetilde{M},0)$ 
             such that the induced mixed {\rm trTLEP}-structure 
             $\mathcal{T}(\mathscr{F})$ gives a universal unfolding of $\mathcal{T}$.             
             \end{theorem}
             This is a generalization of Theorem 4.5 in Hertling-Manin \cite{H.M1}.
             ^^ ^^ Some conditions'' in this theorem is explained 
             in Definition \ref{definition of residually flat section} 
             and Definition \ref{definition of generating condition}.
             We also give the definition of the equivalent condition for a special case in \S\ref{MFS VMHS}.
         \subsection{Mixed Frobenius manifolds and variations of mixed Hodge structure}
         \label{MFS VMHS}
         Consider a graded polarizable variation of mixed Hodge structure (VMHS)
         $\mathscr{H}:=(\mathbb{V}_\mathbb{Q},F,W)$ on a
         germ of a complex manifold  $(M,0)$.
         Here, $\mathbb{V}_\mathbb{Q}$ is a local system of $\mathbb{Q}$-vector space, 
         $F=(F^\ell\mid \ell\in \mathbb{Z})$ is a Hodge filtration
         on $K:=\mathbb{V}_\mathbb{Q}\otimes \mathcal{O}_{M,0}$,
         and $W=(W_k\mid k\in \mathbb{Z})$ is a weight filtration 
         on $\mathbb{V}_\mathbb{Q}$.
         If we fix  a graded polarization $S=(S_k\mid k\in \mathbb{Z})$ on $\mathrm{Gr}^WV_\mathbb{Q}$
         and
         an opposite filtration $U=(U_\ell\mid \ell\in \mathbb{Z})$ 
         (see Definition \ref{opposite filtration} for the definition of opposite filtrations), 
         then we have a mixed trTLEP-structure $\mathcal{T}(\mathscr{H},S,U)$ by Rees construction
         (Lemma \ref{REES}).
         
         For the mixed trTLEP-structure $\mathcal{T}=\mathcal{T}(\mathscr{H},U,S)$, 
         ^^ ^^ some conditions" in Theorem \ref{intro main theorem} can be reformulated as 
         a condition for $\mathscr{H}$. 
         The condition is called $H^2$-generation condition(\cite{H.M1}), motivated by quantum cohomology.
         
         Let $\nabla$ be the flat connection on $K$ induced by $\mathscr{H}$.
         By Griffiths transversality,
         we have a Higgs field 
         $\theta:=\mathrm{Gr}_F(\nabla):\mathrm{Gr}_F K
         \to \mathrm{Gr}_F K\otimes \Omega^1_{M,0}$.
         Put $w:=\max \{\ell\in \mathbb{Z}\mid \mathrm{Gr}_F^\ell K\neq 0\}$.
         Note that $\mathrm{Gr}^w_F K=F^w$. 
         The $H^2$-generation condition for 
         $\mathscr{H}$ is the following.
         \begin{enumerate}
         \item[(i)] The rank of $F^w$ is 1,
                    and the rank of $\mathrm{Gr}^{w-1}_FK$ 
                    is equal to the dimension of $(M,0)$,
         \item[(ii)] The map $\mathrm{Sym}\ \Theta_{M,0}\otimes F^w \to \mathrm{Gr}_F K$ 
                     induced by $\theta$ is surjective.
         \end{enumerate}
         Let $\zeta_0$ be a non-zero vector in $F^w|_0$.
         Then the condition (ii) is equivalent to the condition that $\mathrm{Gr}_F K|_0$ is 
         generated by $\zeta_0$ over $\mathrm{Sym}\ \Theta_{M,0}|_0$
         and conditions (i) and (ii) imply $\Theta_{M,0}\simeq \mathrm{Gr}_F^{w-1}K$.
         
         The following theorem is an application of Theorem \ref{intro main theorem}
         in the case of $\mathcal{T}=\mathcal{T}(\mathscr{H},U,S)$.
         
         \begin{theorem}[Corollary \ref{H2-generated mixed Frobenius}]\label{intro main theorem 2}
         Let $\mathscr{H}=(\mathbb{V}_\mathbb{Q},F,W)$ be a VMHS
         on a germ of a complex manifold $(M,0)$ with the $H^2$-generation condition.
         Take an integer $w$ as above and a non-zero vector $\zeta_0\in F^w\mid_0$.
         Fix a graded polarization $S$ and an opposite filtration $U$ on $\mathrm{Gr}^W V_\mathbb{Q}$.
         Then there exists 
         $($uniquely up to isomorphisms$)$ 
         a tuple $\big{(}(\widetilde{M},0),\mathscr{F},\iota, i\big{)}$ with the following conditions.
         \begin{enumerate}
         \item[$1.$] $\mathscr{F}=(\circ,{\bm \nabla},e,E,\mathcal{I},g)$ 
                     is a MFS on a 
                     germ of a complex manifold $(\widetilde{M},0)$.
         \item[$2.$] $\iota: (M,0)\hookrightarrow (\widetilde{M},0)$ is a closed embedding.
         \item[$3.$] $i:\mathcal{T}(\mathscr{H},U,S)\xrightarrow{\sim} \mathcal{T}(\mathscr{F})$ is
                     an isomorphism of mixed {\rm trTLEP}-structure with
                     $i|_{(0,0)}(\zeta_0)=e|_0$.  
         \end{enumerate}
         \end{theorem} 
           
%
%
         
         \subsection{Mixed Frobenius manifolds in local B models}\label{intro local B}
         We shall explain some applications of construction theorem to the local mirror symmetry.
         Konishi and Minabe introduced the notion of mixed Frobenius manifolds in \cite{konisi1, konisi2}
         to understand the local mirror symmetry. 
         In \cite{konisi1}, they constructed mixed Frobenius manifolds for weak Fano toric surface.
         It remains to construct mixed Frobenius manifolds for local B-models.
         
         Mixed Frobenius manifolds for local B-models are expected to be constructed from 
         variations of mixed Hodge structure (\cite{konisi3}). 
         Using the results of Batyrev \cite{Baty} and Stienstra \cite{J.S}, 
         Konishi and Minabe \cite{konisi3} gave a combinatorial description for 
         the VMHS in the local B-models. 
         
         Let $\Delta\subset \mathbb{Z}^2$ be a two dimensional reflexive polyhedron.
         We have the moduli space $\mathcal{M}(\Delta)$ 
         of affine hypersurfaces in $(\mathbb{C}^\ast)^2$ (Definition \ref{moduli of hypersurface}).
         Let $V_f\subset (\mathbb{C}^\ast)^2$ be a hypersurface corresponds to $[f]\in \mathcal{M}(\Delta)$.
         Fix a stable smooth point $[f_0]$ in $\mathcal{M}(\Delta)$.
         Then the mixed Hodge structure on the relative cohomology 
         $H^2((\mathbb{C}^\ast)^2,V_f),\ (f\in \mathcal{M}(\Delta))$ defines
         a VMHS  
         $\mathscr{H}_\Delta$ on the germ of complex manifold
         $(\mathcal{M}(\Delta),[f_0])$.
         Using the results of $\cite{konisi3}$ and \cite{Baty}, we give a sufficient condition 
         for $\mathscr{H}_\Delta$ to satisfy the $H^2$-generation condition in terms of the toric data. 
         As a consequence, we have
         the following theorem.
         
         \begin{theorem}[\rm Corollary \ref{construction in local B}]\label{intro hodge}
         Fix a graded polarization $S$ and an opposite filtration $U$ for $\mathscr{H}_\Delta$, 
         and take a non-zero vector $\zeta_0\in F^2([f_0])$.
         Then, there exists a tuple $((\widetilde{M},0),\mathscr{F},\iota,i)$ 
         with the following conditions uniquely up to isomorphisms.
         \begin{enumerate}
         \item[$1.$] $\mathscr{F}=(\theta,\nabla,e,E,W,g)$ is a mixed Frobenius structure on $(\widetilde{M},0)$.
         \item[$2.$] $\iota:(\mathcal{M}(\Delta),[f_0])\hookrightarrow (\widetilde{M},0)$ is an embedding.
         \item[$3.$] $i:\mathcal{T}(\mathscr{H}_\Delta,U,S)\xrightarrow{\sim}\iota^*\mathcal{T}(\mathscr{F})$
               is an isomorphism of {\rm trTLEP}-structure with $i|_{(0,[f_0])}(\zeta_0)=e|_0$.
         \end{enumerate}
         \end{theorem}
         This theorem gives the mixed Frobenius manifolds associated to local B-models.
                  
         \subsection{Limit mixed trTLEP-structure and local A-models}\label{intro local A}
         We shall give a method to construct mixed trTLEP-structures from logarithmic trTLEP-structures.
         Let $(M,0)$ be a germ of complex manifold and $(Z,0)\subset (M,0)$ a co-dimension $1$ submanifold.
         If we are given a $\log Z$-trTLEP(0)-structure $\mathcal{T}$ 
         (the definition of $\log Z$-trTLEP-structure is given in Definition
         \ref{logarithmic trTLEP} or \cite[Definition 1.8]{R1})
         satisfies some conditions, we have a mixed trTLEP-structure
         $\mathcal{T}_{Z,0}$ on $(Z,0)$, 
         which is called a limit mixed trTLEP-structure
         (See Definition \ref{LMtrTLEP}).
         
         Let $X$ be a weak Fano toric manifold. Let $r$ be the dimension of  $H^2(X,\mathbb{C})$. 
         For an apporopiate open embedding of $H^2(X,\mathbb{C})/2\pi\sqrt{-1}H^2(X,\mathbb{Z})$ 
         to $\mathbb{C}^r$, 
         we have the logarithmic trTLEP-structure $\mathcal{T}^{\rm small}_X$
         on a neighborhood $V$ of $0\in \mathbb{C}^r$, which is called the small quantum D-module (\cite{RS}).
                
         Let $S$ be a weak Fano toric surface and $X$ be 
         the projective compactification of the canonical bundle of $S$.
         There is a divisor $Z$ of $V$ which is canonically identified with
         a locally closed subset of the quotient space $H^2(S,\mathbb{C})/2\pi\sqrt{-1}H^2(S,\mathbb{Z})$.
         For each $z\in Z$, the logarithmic trTLEP-structure 
         $\mathcal{T}^{\rm small}_X$ on $(V,z)$ induces 
         a limit mixed trTLEP-structure $(\mathcal{T})_{Z,z}$ 
         on $(Z,z)$.
         Moreover, if $z$ is close to $0\in\mathbb{C}^r$ enough, 
         then $(\mathcal{T})_{Z,z}$ induces the mixed Frobenius manifold 
         constructed by Konishi-Minabe \cite{konisi1}.       
        \subsubsection*{Acknowledgement}
        The author would like to express his deep gratitude to his supervisor Takuro Mochizuki 
        for his valuable advice and many suggestions to improve this paper.        
        \section{Construction theorem for mixed Frobenius manifolds}
         The aim of this chapter is to prove the construction theorem for mixed Frobenius manifolds 
         (Corollary \ref{construction theorem of mixed Frobenius manifold}),
         which is a generalization of \cite[Theorem 4.5]{H.M1}.
         
         The symbol
           $\mathbb{P}^1_\lambda$ denotes a projective line
           with non-homogeneous parameter $\lambda$. 
         We identify a holomorphic vector bundle with the associated locally free sheaf.
         For a holomorphic vector bundle $K$, we write $s\in K$ to mean that $s$ is a local section of $K$. 
         We denote the dual vector bundle of $K$ by $K^\vee$. 
         When we consider filtrations, we always assume that the filtrations are exhaustive.
         Hence we always omit ^^ ^^ exhaustive".
          \subsection{Mixed trTLEP-structures and mixed Frobenius manifolds}\label{section correspond trTLE}        
           We define the notion of 
           mixed trTLEP-structures and mixed Frobenius manifolds.
           We show that a mixed Frobenius manifold always induces a mixed trTLEP-structures.
           We also show that a 
           mixed trTLEP-structure induces a mixed Frobenius manifold
           under certain conditions.
 
         \subsubsection{trTLE-structures and Saito structures}
          Recall the definition of trTLE-structures.  
          Let $M$ be a complex manifold and $p_\lambda:\mathbb{P}^1_\lambda\times M\to M$ a
          natural projection.        

        \begin{definition}[{\rm \cite{H.M1,H}}]\label{tle}
                A pair $(\mathcal{H},\nabla)$ is called a {\bf {\rm \bf trTLE}-structure} on $M$
        if the following properties are satisfied:
         \begin{enumerate}
          \item[$1.$] 
                 $\mathcal{H}$ is a holomorphic vector bundle on $\mathbb{P}^1_\lambda \times M$ 
                 such that the adjoint morphism 
                 $p_\lambda^*{p_{\lambda}}_*\mathcal{H}\to \mathcal{H}$ is an isomorphism,  
          \item[$2.$]  $\nabla$ is a meromorphic flat connection on $\mathcal{H}$ 
                 with pole order $1$ along $\{0\}\times M$ 
                 and logarithmic pole along $(\{\infty\}\times M) 
                 :$
                 \begin{equation*}
                  \nabla: \mathcal{H}\to \mathcal{H}\otimes 
                  \Omega^1_{\mathbb{P}^1_{\lambda}\times M}
                  \big{(}\log (\{0,\infty \}\times M)
                  \big{)}
                  \otimes \mathcal{O}_{\mathbb{P}^1_{\lambda}\times M}
                  (\{0\}\times M).
                 \end{equation*}
         \end{enumerate}
         A morphism of trTLE-structures is a flat morphism of holomorphic vector bundles.
        \end{definition}
        
        \begin{remark} 
           For a {\rm trTLE}-structure $(\mathcal{H},\nabla)$ and a complex number $c$, 
            the pair 
            $(\mathcal{H},\nabla+c\cdot \mathrm{id}_\mathcal{H}\ \lambda^{-1}\mathrm{d}\lambda)$ 
            is also a {\rm trTLE}-structure.
        \end{remark}
        
        We recall the definition of Saito structure (without a metric) in \cite{Sab1}.
        Let $M$ be a complex manifold and $p_\lambda:\mathbb{P}^1_\lambda\times M\to M$ a
          natural projection.        
        Suppose that its tangent bundle $\Theta_M$ is equipped with
           a symmetric Higgs field $\theta$,
           a torsion free flat connection ${\bm \nabla}$, and
           two global sections $e$ and $E$.
           We have endomorphisms 
           ${\bm \nabla}_\bullet E$ and $\theta_E$ defined by
           $a\mapsto {\bm \nabla}_aE$ and $a\mapsto\theta_E(a)$ for $a\in \Theta_M$.
           
        \begin{definition}[{\cite[Definition V\hspace{-.1em}I\hspace{-.1em}I. $1.1$]{Sab1}}]        
        \label{def of saito} 
         The tuple $\mathcal{S}:=(\theta,{\bm \nabla}, e, E)$ is called a {\bf Saito structure} on M 
         if the following conditions are satisfied.
          \begin{enumerate}
            \item[$1.$] The vector field $e$ is ${\bm \nabla}$-flat and $\theta_e a=-a$ for all $a\in\Theta_M$.
            \item[$2.$] The following meromorphic connection $\hat{\nabla}$ on $p_\lambda^*\Theta_M$ is flat:
            \begin{equation}\label{hatnabla}
            \hat{\nabla}:=p_\lambda^*{\bm \nabla}+\frac{1}{\lambda}p_\lambda^*\theta-
                          \Big{(}\frac{1}{\lambda}p_\lambda^*\theta_E+p_\lambda^*{\bm \nabla}_\bullet E \Big{)}
                          \frac{\mathrm{d}\lambda}{\lambda}.
            \end{equation}
          \end{enumerate} 
           The vector field $e$ is called the {\bf unit vector field} and $E$ is called the {\bf Euler vector field}.
        \end{definition} 
        
        \begin{remark}\label{remark on saito} 
        \begin{enumerate}
      
         \item[$1.$] The tangent bundle $\Theta_M$ has the structure of 
                     a commutative associative $\mathcal{O}_M$-algebra defined by 
                     $a\circ b:=-\theta_ab\ (a,b\in\Theta_M)$. 
                     The unit vector field $e$ is the global unit section of this algebra.
        \item[$2.$]  The flatness of $\hat{\nabla}$ is equivalent to the condition 
                     that the equations
                     ${\bm \nabla}(\circ)=0$, ${\bm \nabla}({\bm \nabla}_\bullet E)=0$, 
                     and $\mathrm{Lie}_E(\circ)=\circ$ hold.
        \end{enumerate}
        \end{remark} 
        By definition, a Saito structure always induces a trTLE-structure.
        \begin{lemma}\label{lemma Saito to trTLE}
        Let $\mathcal{S}:=(\theta,{\bm \nabla},e,E)$ be a Saito Structure on a complex manifold $M$.
        Then the pair $(p_\lambda^*\Theta_M,\hat{\nabla})$ 
        is a {\rm trTLE}-structure on $M$. 
        \qed \end{lemma}       
        We recall the definition of Frobenius type structure.
         \begin{definition}[{\cite[Definition 5.6]{H}}]\label{Frob type}
          Let $K$ be a holomorphic vector bundle over a complex manifold $M$.
          A {\bf Frobenius type structure} on $K$ consists of 
          a flat connection $\nabla^{\rm r}$ on $K$, 
          a Higgs field $\mathcal{C}$ on $K$,
          and endomorphisms $\mathcal{U},\mathcal{V}\in \mathrm{End}(K)$ such that
          \begin{gather}
          \label{A}
           \nabla^{\rm r}(\mathcal{C})=\nabla^{\rm r}(\mathcal{V})=[\mathcal{C},\mathcal{U}]=0 ,\\
          \label{B}
           \nabla^{\rm r}(\mathcal{U})-[\mathcal{C},\mathcal{V}]+\mathcal{C}=0.
          \end{gather}         
         \end{definition}
       We remark that this definition of Frobenius type structure lacks the pairing 
       comparing with \cite[Definition 5.6]{H}.
       There is a correspondence between a trTLE-structure and a Frobenius type structure as follows.
        \begin{lemma}[{\cite[Theorem 5.7]{H}}]\label{Frob type trTLE}
         Let $(\mathcal{H},\nabla)$ be a trTLE-structure on a complex manifold $M$.
         There exists a unique Frobenius type structure 
         $(\nabla^{\rm r},\mathcal{C},\mathcal{U},\mathcal{V})$ on $\mathcal{H}|_{\lambda=0}$
         such that
         \begin{align} \label{equation nabla}
          \nabla=p_\lambda^*\nabla^{\rm r}+
                \frac{1}{\lambda}p_\lambda^*\mathcal{C}
                +\Big{(}\frac{1}{\lambda}p_\lambda^*\mathcal{U}-p_\lambda^*\mathcal{V}\Big{)}
                \frac{\mathrm{d}\lambda}{\lambda}
         \end{align}
         via the natural isomorphism $\mathcal{H}\simeq p_\lambda^*(\mathcal{H}|_{\lambda=0})$.
         We call it the Frobenius type structure associated to $(\mathcal{H},\nabla)$.
        \end{lemma}
        \begin{proof}
         Let $a$ be a local section of $\Theta_M$. 
         Extend the section $a$ constantly along $\mathbb{P}^1_\lambda$ 
         and denote it by $\widetilde{a}$.
         Similarly, take a local section $s$ of $\mathcal{H}|_{\lambda=0}$
         and extend it to the local section $\widetilde{s}$ of $\mathcal{H}$.
         Define $\mathcal{C}_as$ as 
         the restriction of $\lambda\nabla_{\widetilde{a}}\widetilde{s}$ to $\lambda=0$.
         Define $\mathcal{U}s$ as the restriction of 
         $\lambda\nabla_{\lambda\partial_\lambda}\widetilde{s}$ to $\{\lambda=0\}$.
         Since the flat section $\nabla$ is pole order 1 along $\{\lambda=0\}$, 
         the morphism $(a,s)\mapsto\mathcal{C}_as$ defines a Higgs field on $\mathcal{H}|_{\lambda=0}$
         and $s\mapsto \mathcal{U}s$ defines an endomorphism on $\mathcal{H}|_{\lambda=0}$.         
            
         Since $\nabla$ is regular singular along $\{\lambda=\infty\}$, 
         we have the residual connection $\nabla^{\rm res}$ and  
         the residue endomorphism ${\rm Res}_{\lambda=\infty}\nabla$ on $\mathcal{H}|_{\lambda=\infty}$.
         By the condition 1 in Definition \ref{tle}, 
         we have a natural isomorphism $\mathcal{H}|_{\lambda=\infty}\simeq \mathcal{H}|_{\lambda=0}$.
         Using the isomorphism, regard $\nabla^{\rm res}$ as the connection on $\mathcal{H}|_{\lambda=0}$
         and denote it by $\nabla^{\rm r}$. Similarly, regard 
         ${\rm Res}_{\lambda=\infty}\nabla$
         as a endomorphism on  $\mathcal{H}|_{\lambda=0}$
         and denote it by $\mathcal{V}$.
         One can check the equation (\ref{equation nabla}).
         The flatness of $\nabla$ implies (\ref{A}) and (\ref{B}).
         The uniqueness is trivial by construction.          
        \end{proof}

        \begin{definition}[\cite{H.M1}]\label{definition of residually flat section}
        Let $(\mathcal{H},\nabla)$ be a {\rm trTLE}-structure on $M$.
        Let $(\nabla^{\rm r},\mathcal{C},\mathcal{U},\mathcal{V})$ 
        be the Frobenius type structure associated to $(\mathcal{H},\nabla)$.
        Assume that there is a $\nabla^{\rm r}$-flat global section $\zeta$ of $\mathcal{H}|_{\lambda=0}$.
        \begin{itemize}
        \item
                    The section $\zeta$ is said to satisfy the injectivity condition 
                    $($resp. the identity condition$)$ when
                    the induced morphism
                    \begin{equation}\label{injectively condition}
                      \mathcal{C}_\bullet\zeta:\Theta_M\rightarrow \mathcal{H}|_{\lambda=0}
                    \end{equation}
                   defined by $a\mapsto \mathcal{C}_a\zeta$ $(a\in \Theta_M)$
                   is an injective morphism $($resp. an isomorphism$)$.
        \item
                   Take a complex number $d$.
                   The section $\zeta$ is said to satisfy the eigenvalue condition for $d$
                   $($with respect to $(\mathcal{H},\nabla)$ $)$ if the following equation holds.
                   \begin{equation}\label{eigenvalue condition}
                   \mathcal{V}\zeta=\frac{d}{2}\zeta.
                   \end{equation}
        \end{itemize}
               We denote by {\rm (IC), (IdC)}, and $(\mathrm{EC})_d$ the injectively condition, 
               the identity condition and the eigenvalue condition for $d$ respectively.
        \end{definition}
        \begin{remark}\label{remark shift}
           Let $(\mathcal{H},\nabla)$, $(\nabla^{\rm r},\mathcal{C},\mathcal{U},\mathcal{V})$,
           and $\zeta$ be as in Definition $\ref{definition of residually flat section}$.
           Fix complex numbers $c$ and $d$. 
           The Frobenius type structure associated to 
           $(\mathcal{H},\nabla-c\cdot\mathrm{id}_{\mathcal{H}}\lambda^{-1}\mathrm{d}\lambda)$
           is  $(\nabla^{\rm r},\mathcal{C},\mathcal{U},\mathcal{V}+c\cdot\mathrm{id})$.         
           If $\zeta$ satisfies ${\rm (EC)}_d$ with respect to $(\mathcal{H},\nabla)$, 
           then it satisfies ${\rm (EC)}_{d+2c}$ with respect to  
           $(\mathcal{H},\nabla-c\cdot\mathrm{id}\lambda^{-1}\mathrm{d}\lambda)$.     
         \end{remark}
         Let $\mathcal{S}=(\theta,{\bm \nabla},E,e)$ be a Saito structure on a complex manifold $M$. 
         By Lemma \ref{lemma Saito to trTLE}, we have the trTLE-structure 
         $(p_\lambda^*\Theta_M,\hat{\nabla})$. 
         Comparing the equations (\ref{hatnabla}) and (\ref{equation nabla}),
         we can check that
         the Frobenius type structure associated to $(p_\lambda^*\Theta_M,\hat{\nabla})$ is 
         $({\bm \nabla},\theta,\theta_E,{\bm \nabla}_\bullet E)$.
         \begin{lemma}\label{IdC for Saito}
         The unit vector field $e$ satisfies {\rm (IC)} and $(\mathrm{EC})_2$.
         \end{lemma}
         \begin{proof}
         The condition 1 in Definition \ref{def of saito} implies 
         that the unit vector field $e$ satisfies (IC).
         Since ${\bm \nabla}$ is torsion free and $e$ is ${\bm \nabla}$-flat, 
         ${\bm \nabla}_eE={\bm \nabla}_Ee-[E,e]=-[E,e]$.
         By Remark \ref{remark on saito}, 
         $$[E,e\circ e]-[E,e]\circ e-e\circ[E,e]=e\circ e.$$
         This implies $-[E,e]=e$.
         Hence $e$ also satisfies $(\mathrm{EC})_2$.
         \end{proof}
        Let $(\mathcal{H},\nabla)$ be a {\rm trTLE}-structure on $M$.
        Let $(\nabla^{\rm r},\mathcal{C},\mathcal{U},\mathcal{V})$ 
        be the Frobenius type structure associated to $(\mathcal{H},\nabla)$. 
        Assume that we have a global $\nabla^{\rm r}$-flat section $\zeta$ of $\mathcal{H}|_{\lambda=0}$
        with (IdC) and $(\mathrm{EC})_2$.
        Put $\mu:=-\mathcal{C}_\bullet\zeta:\Theta_M\xrightarrow{\sim}\mathcal{H}|_{\lambda=0}$.
        Using this isomorphism, regard $\nabla$ (resp. $\mathcal{C}$) 
        as a flat section (resp. Higgs field) on $\Theta_M$ and denote it by the same letter.
        Put $e:=\mu^{-1}(\zeta)$ and $E:=\mu^{-1}(\mathcal{U}\zeta)$.
        The following proposition is essentially proved in \cite{H.M1} and \cite{R1}.
        \begin{proposition}\label{lemma trTLE to Saito}
          \begin{itemize}
           \item The tuple $\mathcal{S}_{\mathcal{H},\zeta}:=(\mathcal{C},\nabla^{\rm r},e,E)$
                 is a Saito structure on $M$.
           \item $\mathcal{S}_{\mathcal{H},\zeta}$ is the unique Saito structure on $M$
                 such that $\mu(e)=\zeta$ and the morphism
                 $$p^*_\lambda (\mu):p^*_\lambda \Theta_M\xrightarrow{\sim}p^*_\lambda(\mathcal{H}|_{\lambda=0})
                 \simeq \mathcal{H}$$
                 gives an isomorphism of {\rm trTLE}-structures
                 between $(p^*_\lambda\Theta_M,\hat{\nabla})$ and $(\mathcal{H},\nabla)$.
                 \qed
          \end{itemize}
         \end{proposition}
         We conclude this subsection with the following corollary.
         \begin{corollary}\label{saito trTLE}
          Fix a complex number $d$. 
          Put $c:=(2-d)/2$. 
          Let $M$ be a complex manifold and 
          $p_\lambda:\mathbb{P}^1_\lambda\times M\to M$ a natural projection.        
          \begin{enumerate}
           \item[$1.$] If $\mathcal{S}$ is a Saito structure on $M$,
                       then $\mathcal{H}_{\mathcal{S},d}
                       :=(p_\lambda^*\Theta_M,\hat{\nabla}+c\cdot \mathrm{id} \lambda^{-1}d\lambda)$
                       is a {\rm trTLE}-structure such that the unit vector field 
                       satisfies {\rm (IdC)} and $(\mathrm{EC})_d$. 
           \item[$2.$] Let $(\mathcal{H},\nabla)$ be a trTLE-structure and 
                       $(\nabla^{\rm r},\mathcal{C},\mathcal{U},\mathcal{V})$ 
                       the associated Frobenius type structure. 
                       Let $\zeta$ be a $\nabla^{\rm r}$-flat global section of $\mathcal{H}|_{\lambda=0}$
                       satisfying {\rm (IdC)} and $(\mathrm{EC})_d$.
                       Then there is a unique Saito structure $\mathcal{S}$ on $M$ such that 
                       the unit vector $e$ satisfies
                       $-\mathcal{C}_e\zeta=\zeta$ and
                       the morphism 
                       $$-p^*_\lambda(\mathcal{C}_\bullet\zeta):p_\lambda^*\Theta_M\xrightarrow{\sim}
                       p^*_\lambda(\mathcal{H}|_{\lambda=0})\simeq\mathcal{H}$$
                       gives an isomorphism of {\rm trTLE}-structures 
                       between $\mathcal{H}_{\mathcal{S},d}$ and $(\mathcal{H},\nabla)$.                                                                                          
          \end{enumerate}
         \end{corollary}
         \begin{proof}
          The first assertion is easily checked by using Lemma \ref{IdC for Saito} and Remark \ref{remark shift}. 
          Let $(\mathcal{H},\nabla)$ and $\zeta$ be the same as in the second assertion.
          By Remark \ref{remark shift}, $\zeta$ satisfies (IdC) and $\mathrm{(EC)}_2$ with respect to 
          $(\mathcal{H},\nabla-c\cdot \mathrm{id}_\mathcal{H}\lambda^{-1}\mathrm{d}\lambda)$.
          Hence by Proposition \ref{lemma trTLE to Saito},
          there exists a unique Saito structure $\mathcal{S}$
          such that $-\mathcal{C}_e\zeta=\zeta$ and 
          $-p^*_\lambda(\mathcal{C}_\bullet\zeta)$ gives an 
          isomorphism of trTLE-structures between 
          $(p_\lambda^*\Theta_M,\hat{\nabla})$ and 
          $(\mathcal{H},\nabla-c\cdot \mathrm{id}_\mathcal{H}\lambda^{-1}\mathrm{d}\lambda)$.          
         \end{proof}
        
        \subsubsection{Weight filtrations, graded pairings, and Frobenius filtrations}
        Let $M$ be a complex manifold.       
         Let $j_\lambda:\mathbb{P}^1_\lambda\times M\to \mathbb{P}^1_\lambda\times M$ be the morphism
         defined by $j_\lambda(\lambda,t)=(-\lambda,t)$ where $\lambda$ is the non-homogeneous coordinate
         on $\mathbb{P}^1_\lambda$ and $t$ is a point in $M$.
         For two holomorphic vector bundles $\mathcal{E}$ and $\mathcal{F}$
         , we denote by $\sigma$ the natural isomorphism 
         $\mathcal{E}\otimes\mathcal{F}\to\mathcal{F}\otimes\mathcal{E}$ given by
         $$e\otimes f\mapsto f\otimes e\ \ (e\in\mathcal{E},\ f\in\mathcal{F}).$$         
        For an integer $k$,  
        we denote the invertible sheaf 
        $\mathcal{O}_{\mathbb{P}^1_\lambda\times M}(-k\{0\}\times M + k\{\infty\}\times M)$
         by $\lambda^k\mathcal{O}_{\mathbb{P}^1_\lambda\times M}$.
         Let $\mathcal{H}$ be a holomorphic vector bundle on $\mathbb{P}^1_\lambda\times M$.        
        Let $P: \mathcal{H}\otimes j_\lambda^*\mathcal{H}\to \lambda^k\mathcal{O}_{\mathbb{P}^1_\lambda\times M}$
        be a morphism of $\mathcal{O}_{\mathbb{P}^1_\lambda\times M}$-modules.
         The morphism $P$ is called $(-1)^k$-symmetric
        if $j_\lambda^*P=(-1)^kP\circ\sigma$, 
        and it is called non-degenerate if 
        the morphism $\mathcal{H}\to (j^*_\lambda\mathcal{H})^{\vee}$ induced by 
        $\lambda^{-k}P$ is an isomorphism.

         Recall the definition of trTLEP($k$)-structures.
        \begin{definition}[{\rm \cite{H}}]\label{def of trTLEP(k)}
          Let $(\mathcal{H},\nabla)$ be a trTLE-structure on $M$ and fix an integer $k$.
          If a morphism of $\mathcal{O}_{\mathbb{P}^1_\lambda\times M}$-modules
          $P: \mathcal{H}\otimes j_\lambda^*\mathcal{H}\to \lambda^k\mathcal{O}_{\mathbb{P}^1_\lambda\times M}$ 
          is 
          $\nabla$-flat,
          $(-1)^k$-symmetric, and  
          non-degenerate,        
          then we call the triple $(\mathcal{H},\nabla,P)$ a {\bf{\rm \bf trTLEP}$(k)$-structure}. 
          We also call the morphism $P$ a {\bf pairing} of the {\rm trTLEP}$(k)$-structure.
         \end{definition}
        We introduce the notions of filtered trTLEP-structures and mixed trTLEP-structures.
        \begin{definition}\label{def mixed trTLEP}
        Let $(\mathcal{H},\nabla)$ be a {\rm trTLE}-structure on $M$.
              \begin{enumerate}
               \item[$1.$] An increasing filtration $W=(W_k\mid k\in \mathbb{Z})$ 
                           of 
                           $\nabla$-flat subbundles in $\mathcal{H}$
                           is called a {\bf weight filtration} of $(\mathcal{H},\nabla)$ 
                           if the subquotient
                           $\mathrm{Gr}^W_k\mathcal{H}:=W_k/W_{k-1}$ 
                           is a {\rm trTLE}-structure for 
                           every integer $k$. 
                           We call the triple $\mathcal{T}_{\rm filt}:=(\mathcal{H},\nabla,W)$ 
                           a {\bf filtered {\rm \bf trTLE}-structure} if
                           $(\mathcal{H},\nabla)$ is a {\rm trTLE}-structure 
                           and $W$ is a weight filtration of it.
               \item[$2.$] Let $W$ be a weight filtration of $(\mathcal{H},\nabla)$.
                           A sequence of morphisms
                           $$P:=
                           (P_k:\mathrm{Gr}^W_k\mathcal{H}\otimes j^*_\lambda \mathrm{Gr}^W_k\mathcal{H}
                           \to \lambda^{-k}\mathcal{O}_{\mathbb{P}^1_\lambda\times M}\mid k\in\mathbb{Z})$$
                           is called a {\bf graded pairing} 
                           on the filtered {\rm trTLE}-structure 
                           $\mathcal{T}_{\rm filt}=(\mathcal{H},\nabla,W)$
                           if the triple $(\mathrm{Gr}_k^W(\mathcal{H}),\nabla,P_k)$ is
                           a {\rm trTLEP}$(-k)$-structure for every integer $k$. 
                           We call the pair 
                           $\mathcal{T}:=(\mathcal{T}_{\rm filt},P)$
                           a {\bf mixed trTLEP-structure} 
                           if $\mathcal{T}_{\rm filt}$ is a filtered {\rm trTLE}-structure and
                           $P$ is a graded pairing on it.
              \end{enumerate}
            Isomorphisms of these structures are the isomorphisms of underlying trTLE-structures which 
            preserves the weight filtrations and graded Pairings.  
        \end{definition} 
        \begin{remark}\label{remark on MtrTLEP}
        \begin{itemize}
        \item
               We can define pull-backs for mixed trTLEP-structures. 
              Let $f:M_0\to M_1$ be a holomorphic map 
              and $\mathcal{T}=(\mathcal{H}, (W_k)_k, (P_k)_k)$ a mixed trTLEP-structure on $M_1$.
              Then $f^*\mathcal{T}:=\big{(}(\mathrm{id}_{\mathbb{P}^1_\lambda}\times f)^*\mathcal{H},
                                     \{(\mathrm{id}_{\mathbb{P}^1_\lambda}\times f)^*W_k\}_k,
                                     \{(\mathrm{id}_{\mathbb{P}^1_\lambda}\times f)^*P_k\}_k
                                          \big{)}$   
              is mixed trTLEP-structure on $M_0$.
        \item 
              Let $\mathcal{T}=(\mathcal{H},\nabla, W, P)$ a mixed trTLEP-structure and
              $\ell$ a half-integer. 
              Then if we put $W(\ell)_k:=W_{k+2\ell}$ and $P(\ell)_k:=\lambda^{2\ell}P_{k+2\ell}$, the tuple 
         $\big{(}\mathcal{H},\nabla-\ell\mathrm{id}_\mathcal{H}\lambda^{-1}\mathrm{d}\lambda, W(\ell),P(\ell)\big{)}$
              is a mixed trTLEP structure.
              We denote it by $\mathcal{T}(\ell)$ and call it the Tate twist of $\mathcal{T}$ by $\ell$.
        \item A {\rm trTLEP}$(k)$-structure $\mathcal{T}$ can be regarded as 
              a mixed trTLEP-structure by the canonical way.
              The Tate twist $\mathcal{T}(\ell)$ gives a {\rm trTLEP}$(k+2\ell)$-structure 
              for every half-integer $\ell$. 
        \end{itemize}
        \end{remark}
        We recall the definition of mixed Frobenius structure introduced in \cite{konisi1,konisi2}.        
        Let $\mathcal{S}=(\theta,{\bm \nabla},e,E)$ be a Saito structure on a complex manifold $M$.
        A subbundle $\mathcal{J}\subset\Theta_M$ is called 
        $\theta$-invariant (resp.${\bm \nabla}$-flat) if 
        $\theta_x y\in\mathcal{J}$ (resp.${\bm \nabla}_x y\in\mathcal{J}$) for all $x\in \Theta_M, y\in \mathcal{J}$.
        A subbundle $\mathcal{J}\subset\Theta_M$ is called $E$-closed if 
        $\mathrm{Lie}_E(y)=[E,y]\in \mathcal{J}$ for all $y\in \mathcal{J}$.
        If $\mathcal{J}$ is ${\bm \nabla}$-flat,
        the $E$-closedness of $\mathcal{J}$ is equivalent to 
        the condition that ${\bm \nabla}_yE\in\mathcal{J}$ for all $y\in \mathcal{J}$.
                
        \begin{definition}[{\cite[Definition 6.2]{konisi1},\cite[Definition 4.5]{konisi2}}]\label{def mixed Frob}
        
        Fix a complex number $d$.  
         Let $\mathcal{I}=(\mathcal{I}_k \subset \Theta_M \mid k\in \mathbb{Z})$ 
         be an increasing filtration of 
        $\theta$-invariant, 
        ${\bm \nabla}$-flat, and   
        $E$-closed 
         subbundles on $\Theta_M$.
          Let         
         $g=(
         g_k:\mathrm{Gr}_k^\mathcal{I}\Theta_M\otimes\mathrm{Gr}_k^\mathcal{I}\Theta_M
             \to 
             \mathcal{O}_M
             \mid k\in\mathbb{Z} 
              )$ 
         be a sequence of ${\bm \nabla}$-flat symmetric morphisms.
          The pair $(\mathcal{I},g)$
         is called  a {\bf Frobenius filtration} on the Saito structure $\mathcal{S}$ 
         of charge $d$
         if they satisfy the following equations for all integer $k$:

         \begin{align}\label{g1}
           g_k(\theta_x y,z)&=g_k(y,\theta_x z) \ \ 
           (x\in \Theta_M, y,z\in \mathrm{Gr}^\mathcal{I}_k\Theta_M)\\ 
           \label{g2}  
           \mathrm{Lie}_E(g_k)&=(2-d+k)g_k.        
         \end{align}
         A triple $\mathscr{F}:=(\mathcal{S},\mathcal{I},g)$ is called a {\bf mixed Frobenius structure} 
         $(${\bf MFS}$)$ 
         on a complex manifold $M$ 
         if 
         $\mathcal{S}$ is a Saito structure on $M$ 
         and $(\mathcal{I},g)$ is a Frobenius filtration on $\mathcal{S}$. 
         A complex manifold equipped with a MFS is called {\bf mixed Frobenius manifold}.
        \end{definition}
        \begin{remark} 
        \begin{itemize}
         \item This definition of Frobenius filtrations $($and MFS$)$ is slightly different from 
               the one in {\rm \cite{konisi1,konisi2}}.
               A MFS $(\mathcal{S},\mathcal{I},g)$ gives a Frobenius structure 
               if $\mathcal{I}_{-1}=0$ and $\mathcal{I}_0=\Theta_M$. 
         \item Let $\mathscr{F}=(\mathcal{S},\mathcal{I},g)$ be a MFS of charge $d$ and $\ell$ a half-integer.
               If we put $\mathcal{I}(\ell)_k:=\mathcal{I}_{k+2\ell}$ and $g(\ell)_k:=g_{k+2\ell}$, 
               then $\mathscr{F}(\ell):=(\mathcal{S},\mathcal{I}(\ell), g(\ell))$ is a MFS of charge $d-2\ell$.
        \end{itemize}
        \end{remark} 
        Let $\mathscr{F}=(\mathcal{S},\mathcal{I},g)$ be 
        a MFS
        on a complex manifold $M$ 
        of charge $d$.
        The underlying Saito structure $\mathcal{S}$ induces 
        a trTLE-structure $\mathcal{H}_{\mathcal{S},d}$ (Corollary \ref{saito trTLE}).        
         Put $W_{\mathcal{I},k}:=p_\lambda^*\mathcal{I}_k$ 
         and $W_\mathcal{I}:=(W_{\mathcal{I},k}\mid k\in\mathbb{Z})$.
         \begin{lemma}
          $W_\mathcal{I}$ is a weight filtration on $\mathcal{H}_{\mathcal{S},d}$.
         \end{lemma}
         \begin{proof}
          Recall that $\mathcal{H}_{\mathcal{S},d}=(p_\lambda^*\Theta_M,\hat{\nabla}+c\lambda^{-1}\mathrm{d}\lambda)$
          where $\hat{\nabla}$ is given in (\ref{hatnabla}) and $c=(2-d)/2$. 
          Since each $\mathcal{I}_k$ is $\theta$-invariant, $\nabla$-flat, and $E$-closed,
          $W_{\mathcal{I},k}=p^*_\lambda\mathcal{I}_k$ is 
          ($\hat{\nabla}+c\cdot\mathrm{id}_{p^*_\lambda \Theta_M} \lambda^{-1}\mathrm{d}\lambda$)-flat. 
          The natural isomorphism 
          $\mathrm{Gr}^{W_\mathcal{I}}_k(p^*_\lambda\Theta_M)
          \simeq 
          p^*_\lambda\mathrm{Gr}^\mathcal{I}_k\Theta_M$ 
          shows that $W_\mathcal{I}$ is a weight filtration.
         \end{proof}
         We denote $\mathrm{Gr}^{W_\mathcal{I}}_k(p^*_\lambda\Theta_M)$ 
         by $\mathrm{Gr}_k^{W_\mathcal{I}}(\mathcal{H}_{\mathcal{S},d})$.
         For each integer $k$, 
         let $P_{g,k}$ be a morphism given by the following composition:
         \begin{equation*} 
                       \mathrm{Gr}_k^{W_\mathcal{I}}(\mathcal{H}_{\mathcal{S},d})
                       \otimes j_\lambda^*\mathrm{Gr}_k^{W_\mathcal{I}}
                       (\mathcal{H}_{\mathcal{S},d})                      
                       \xrightarrow{id\otimes j^*}
                       \mathrm{Gr}_k^{W_\mathcal{I}}(\mathcal{H}_{\mathcal{S},d})
                       \otimes \mathrm{Gr}_k^{W_\mathcal{I}}(\mathcal{H}_{\mathcal{S},d})                      
                       \xrightarrow{\lambda^{-k}p_\lambda^*(g_k)}
                       \lambda^{-k}\mathcal{O}_{\mathbb{P}^1_\lambda\times M}.                     
         \end{equation*} 
        \begin{proposition}\label{proposition mixed Frobenius and mixed trTLEP}
         Put $P_g:=(P_{g,k}\mid k\in\mathbb{Z})$. 
         Then $\mathcal{T}(\mathscr{F}):=((\mathcal{H}_{\mathcal{S},d},W_\mathcal{I}),P_g)$
         is a mixed trTLEP-structure.       
        \end{proposition}
        \begin{proof}
        We need to show that $(\mathrm{Gr}^{W_\mathcal{I}}_k(\mathcal{H}_{\mathcal{S},d}),P_{g,k})$ 
        is a trTLEP$(-k)$-structure for each integer $k$.
        Since $g_k$ is symmetric and non-degenerate (and by construction), 
        $P_{g,k}$ is $(-1)^{-k}$-symmetric and non-degenerate. 
        Hence it remains to show that $P_{g,k}$ is 
        ($\hat{\nabla}+c\cdot\mathrm{id}_{p^*_\lambda \Theta_M}\lambda^{-1}\mathrm{d}\lambda$)-flat.        
        This follows from the equations (\ref{g1}) and (\ref{g2}),
        the fact that $g_k$ is $\nabla$-flat, 
        and some easy calculations.
        \end{proof}
         The following proposition describes a correspondence between mixed Frobenius structures and 
         mixed trTLEP-structures. 
        \begin{proposition}
         Let $\mathcal{T}=(\mathcal{H},\nabla,W,P)$ be a mixed trTLEP-structure 
         on a complex manifold $M$.
         Let $(\nabla^{\rm r},\mathcal{C},\mathcal{U},\mathcal{V})$ be the Frobenius type structure 
         associated to the underlying trTLE-structure $(\mathcal{H},\nabla)$.
         Assume that there is a $\nabla^{\rm r}$-flat global section $\zeta$ of $\mathcal{H}|_{\lambda=0}$
         satisfying {\rm (IdC)} and $(\mathrm{EC})_d$ for a complex number $d$.
         Then there exists a unique MFS $\mathscr{F}$ on $M$ of charge $d$ such that 
         the unit vector field $e$ satisfies $-\mathcal{C}_e\zeta=\zeta$ and
         the morphism 
         \begin{align}\label{mu3}
          -p^*_\lambda(\mathcal{C}_\bullet\zeta):p_\lambda^*\Theta_M\xrightarrow{\sim}
                       p^*_\lambda(\mathcal{H}|_{\lambda=0})\simeq\mathcal{H}
         \end{align}
         gives an isomorphism of mixed {\rm trTLEP}-structures 
         between $\mathcal{T}(\mathscr{F})$ and $\mathcal{T}$.
        \end{proposition}
        \begin{proof}
         By Corollary (\ref{saito trTLE}), we have a unique Saito structure 
         $\mathcal{S}$
         such that (\ref{mu3}) gives an isomorphism between 
         $\mathcal{H}_{\mathcal{S},d}$ and $(\mathcal{H},\nabla)$.
         Put $\mu:=-\mathcal{C}_\bullet\zeta$
         and $\mathcal{I}_{W,k}:=\mu^{-1}(W_k|_{\lambda=0})$.
         Since $W_k$ is $\nabla$-flat, $W_k|_{\lambda=0}$ is 
         $\nabla^{\rm r}$-flat, $\mathcal{C}$-invariant, and closed under $\mathcal{U}$ and $\mathcal{V}$.
         Recall that the flat connection and Higgs field of $\mathcal{S}$
         are identified with $\nabla^{\rm r}$ and $\mathcal{C}$ via $\mu$.
         If $E$ is the Euler vector field of $\mathcal{S}$
         and $c=(2-d)/2$,
         the endomorphism $\nabla^{\rm r}_\bullet E$ 
         is identified with 
         $\mathcal{V}+c\cdot\mathrm{id}_{\Theta_M}$ via $\mu$.
         Hence $\mathcal{I}_{W,k}$ is $E$-closed.
         Let $g_{P,k}$ be the restriction of 
         $\lambda^k P$ to $\{\lambda=0\}$
         and put $g_P:=(g_{P,k}\mid k\in \mathbb{Z})$.
         Then via (\ref{mu3}) and the natural isomorphism $(p^*_\lambda\Theta_M)|_{\lambda=0}\simeq\Theta_M$,
         $g_{P,k}$ gives a symmetric, $\nabla^{\rm r}$-flat, non-degenerate pairing 
         on $\mathrm{Gr}^{\mathcal{I}}_k\Theta_M$ satisfying
         (\ref{g1}) and (\ref{g2}), which we denote by the same letter. 
         As a conclusion, the pair $(\mathcal{I}_W,g_P)$ is 
         a Frobenius filtration on $\mathcal{S}$ 
         and hence $\mathscr{F}:=(\mathcal{S},\mathcal{I}_W,g_P)$ is a MFS on $M$.
         It is easy to check that (\ref{mu3}) gives the isomorphism $\mathcal{T}(\mathscr{F})\simeq\mathcal{T}$.
         The uniqueness of such $\mathscr{F}$ is trivial by construction.
        \end{proof}
      
      \subsection{Unfoldings and the construction theorem}
        For a complex manifold $M$ and a point $0\in M$,
        we denote by $(M,0)$ the germ of manifold around $0$.
        Let $(\mathcal{H},\nabla)$ be a trTLE-structure on $(M,0)$,
        and $(\nabla^{\rm r},\mathcal{C},\mathcal{U},\mathcal{V})$ 
        the associated Frobenius type structure on $\mathcal{H}|_{\lambda=0}$(Lemma \ref{Frob type trTLE}).
        Let $\zeta$ be a $\nabla^{\rm r}$-flat section of $\mathcal{H}|_{\lambda=0}$.
        Then $\zeta$ satisfies (IC) (resp. (IdC)) if and only if 
        the map (\ref{injectively condition}) is injective (resp. isomorphism) at $0\in (M,0)$.
        
        Let $\zeta_0$ be a non-zero vector in $\mathcal{H}|_{(0,0)}$
        and take the $\nabla^{\rm r}$-flat extension $\zeta\in\mathcal{H}|_{\lambda=0}$.
        The vector $\zeta_0$ is said to satisfy (IC) (resp. (IdC)) 
        if $\zeta$ satisfies (IC) (resp. (IdC)).
        \subsubsection{Statements of the unfolding theorem and the construction theorem}
        We define the category of unfoldings of mixed trTLEP-structures.
        \begin{definition}[{cf. \cite[Definition 2.3]{H.M1}}]\label{def of unfolding}
        Fix a mixed {\rm trTLEP}-structure $\mathcal{T}$ on $(M,0)$.
        \begin{enumerate}
        \item[$(a)$] An {\bf unfolding} of $\mathcal{T}$ is 
                   a mixed {\rm trTLEP}-structure $\widetilde{\mathcal{T}}$ on
                   a germ
                   $(\widetilde{M},0)$ of complex manifold
                   together with a closed embedding $\iota:(M,0)\to (\widetilde{M},0)$
                   and an isomorphism 
                   $i:\mathcal{T}\simeq \iota^*\widetilde{\mathcal{T}}$. 
        \item[$(b)$] Let $\big{(}(\widetilde{M},0),\widetilde{\mathcal{T}},\iota,i\big{)}$ 
                     and $\big{(}(\widetilde{M}',0),\widetilde{\mathcal{T}}',\iota',i'\big{)}$ 
                     be two unfoldings of $\mathcal{T}$.
                     A {\bf morphism of unfoldings} of $\mathcal{T}$ from 
                     $\big{(}(\widetilde{M},0),\widetilde{\mathcal{T}},\iota,i\big{)}$                                          
                     to $\big{(}(\widetilde{M}',0),\widetilde{\mathcal{T}}',\iota',i'\big{)}$
                     is a pair $(\varphi,\phi)$
                     such that
                     \begin{itemize}
                      \item $\varphi:(\widetilde{M},0) \to (\widetilde{M}',0)$ is a holomorphic map with 
                            $\varphi\circ\iota=\iota'$, and
                      \item $\phi:\widetilde{\mathcal{T}}\xrightarrow{\sim} \varphi^*\widetilde{\mathcal{T}}'$ is an 
                            isomorphism of mixed trTLEP-structure with $i={\iota'}^*(\phi)\circ i'$.
                     \end{itemize}
                     
        \end{enumerate}
        We denote the category of unfoldings by $\mathfrak{Unf}_\mathcal{T}$. 
        A terminal object of $\mathfrak{Unf}_\mathcal{T}$ is called a universal unfolding of $\mathcal{T}$
        if it exists.       
        \end{definition}
        \begin{remark}\label{normalize of unfoldings}
         Let $\big{(}(\widetilde{M},0),\widetilde{\mathcal{T}},\iota,i\big{)}$ be an unfolding such that 
         $(\widetilde{M},0)=\big{(}M\times \mathbb{C}^l,(0,0)\big{)}$ and $\iota$ is the inclusion. 
         Then we denote the unfolding by $\big{(}(M\times\mathbb{C}^l,0),\widetilde{\mathcal{T}},i\big{)}$.
         Every unfolding is isomorphic to such an unfolding. 
        \end{remark}
        Let $(\mathcal{H},\nabla)$ be a trTLE-structure on $(M,0)$,
        and $(\nabla^{\rm r},\mathcal{C},\mathcal{U},\mathcal{V})$
        the associated Frobenius type structure. 
        Recall that 
        $\mathcal{C}$ is a Higgs field and $\mathcal{U}$ is an endomorphism on $\mathcal{H}|_{\lambda=0}$.
        Let $\mathcal{A}$ be the sub-algebra in $\mathrm{End}(\mathcal{H}|_{\lambda=0})$ generated by 
        $\{\mathcal{C}_x\mid x\in \Theta_{M,0}\}$ and $\mathcal{U}$.
        The relation (\ref{A}) implies that $\mathcal{A}$ is a commutative algebra.
         
        \begin{definition}\label{definition of generating condition}
        A $\nabla^{\rm r}$-flat section $\zeta$ is said to satisfy the {\bf generation condition} 
        {\rm ((GC)} for short$)$
        if it generates $\mathcal{H}|_{\lambda=0}$ over $\mathcal{A}$,
        i.e. $\mathcal{A}\zeta=\mathcal{H}|_{\lambda=0}$.
        \end{definition}
        
        \begin{remark}\label{remark on GC}
           A $\nabla^{\rm r}$-flat section $\zeta$ satisfies {\rm (GC)} if and only if 
           its restriction to $(0,0)\in\mathbb{P}^1_\lambda\times M$ 
           generates $\mathcal{H}|_{(0,0)}$ over $\mathcal{A}|_{(0,0)}$.
           A non-zero vector $\zeta_0\in\mathcal{H}|_{(0,0)}$ is said to satisfy {\rm (GC)}
           if its $\nabla^{\rm r}$-flat extension $\zeta\in\mathcal{H}|_{\lambda=0}$ satisfies {\rm (GC)}. 
        \end{remark}
        \begin{lemma}
               If a $\nabla^{\rm r}$-flat section $\zeta$ satisfies {\rm (GC)}, 
               then the map $\mathcal{A}\rightarrow \mathcal{H}|_{\lambda=0},\ a\mapsto a(\zeta)$ 
               is an isomorphism of $\mathcal{O}_{M,0}$-modules.
        \end{lemma}
        \begin{proof}
        Let $a,a'$ be two endomorphisms in $\mathcal{A}$ such that $a(\zeta)=a'(\zeta)$.
        For any section $s\in\mathcal{H}|_{\lambda=0}$, 
        we have $b\in\mathcal{A}$ with $b(\zeta)=s$ by (GC).
        Since $\mathcal{A}$ is a commutative algebra, we have $a(b(\zeta))=b(a(\zeta))$.
        This implies $a(s)=a'(s)$ and hence $a=a'$.         
        \end{proof}
        
        \begin{remark}\label{conditions unfolding}
        Let $\big{(}(\widetilde{M},0),\widetilde{\mathcal{T}},\iota,i\big{)}$ 
        be an unfolding of a mixed {\rm trTLEP}-structure $\mathcal{T}$ on $(M,0)$.
        Let $(\widetilde{\mathcal{H}},\widetilde{\nabla})$ 
        $($resp. $(\mathcal{H},\nabla)$$)$ 
        be the underlying trTLE-structures
        of $\widetilde{\mathcal{T}}$ $($resp. $\mathcal{T}$$)$. 
        The restriction of $i$ to $(0,0)\in\mathbb{P}^1_\lambda\times (M,0)$
        is an isomorphism 
        $i|_{(0,0)}:\mathcal{H}|_{(0,0)}\xrightarrow{\sim}\widetilde{\mathcal{H}}|_{(0,0)}$
        of vector spaces.
        If a non-zero vector $\zeta_0\in\mathcal{H}|_{(0,0)}$
        satisfies {\rm (GC)} or $\mathrm{(EC)}_d$, then $i|_{(0,0)}(\zeta_0)$ satisfies the same condition.
        On the other hand, even if $\zeta_0$ satisfies {\rm (IC)} or {\rm (IdC)}, 
        $i|_{(0,0)}(\zeta_0)$ does not necessarily satisfy the same condition.
        \end{remark}
                
        The following theorem is the first main theorem of this paper.
        \begin{theorem}[Unfolding theorem]\label{unfolding theorem for mixed trTLEP-structure}
        Let $\mathcal{T}$ be a mixed {\rm trTLEP}-structure on a germ $(M,0)$ of complex manifold 
        and $(\mathcal{H},\nabla)$ the underlying {\rm trTLE}-structure. 
        Let $\zeta_0$ be a non-zero vector in $\mathcal{H}|_{0,0}$ 
        satisfying {\rm (IC)} and {\rm (GC)}. 
        Then there exists a universal unfolding of $\mathcal{T}$.
        Moreover, an unfolding $\big{(}(\widetilde{M},0),\widetilde{\mathcal{T}},\iota,i\big{)}$ is universal 
        if and only if the vector
        $i|_{(0,0)}(\zeta_0)$ satisfies {\rm (IdC)}. 
        \end{theorem}
        This theorem will be proved in \S \ref{pf of unf}.
        Using this theorem, 
        we have the following.
        \begin{corollary}[Construction theorem]\label{construction theorem of mixed Frobenius manifold}
        Let $\mathcal{T}$, $(\mathcal{H},\nabla)$, and $\zeta_0$ be the same as in 
        {\rm Theorem \ref{unfolding theorem for mixed trTLEP-structure}}. 
        Assume moreover that $\zeta_0$ satisfies ${\rm (EC)}_d$ for a complex number $d$. 
        Then there exists a tuple $\big{(}(\widetilde{M},0),\mathscr{F},\iota,i\big{)}$ 
        with the following properties uniquely up to isomorphisms.
        \begin{enumerate}
        \item[$1.$] $\mathscr{F}$ is a MFS 
              of charge $d$ on a germ $(\widetilde{M},0)$ of a complex manifold. 
        \item[$2.$] $\iota:(M,0)\hookrightarrow (\widetilde{M},0)$ is a closed embedding.
        \item[$3.$] $i:\mathcal{T}\xrightarrow{\sim} \iota^*\mathcal{T}(\mathscr{F})$ is an isomorphism of 
              mixed {\rm trTLEP}-structure such that $e(0)=i|_{(0,0)}(\zeta_0)$.
        \end{enumerate}        
        \end{corollary}
        \begin{proof}
         By Theorem \ref{unfolding theorem for mixed trTLEP-structure}, 
         we have a universal unfolding 
         $\big{(}(\widetilde{M},0),\widetilde{\mathcal{T}},\iota,i\big{)}$
         of $\mathcal{T}$.
         Since $i|_{(0,0)}(\zeta_0)$ satisfies (IdC) and 
         $\mathrm{(EC)}_d$ (Remark \ref{conditions unfolding}),
         there is a unique MFS $\mathscr{F}$
         on $(\widetilde{M},0)$ such that the 
         morphism (\ref{mu3}) gives the isomorphism $\mathcal{T}(\mathscr{F})\simeq\widetilde{\mathcal{T}}$
         (Proposition \ref{proposition mixed Frobenius and mixed trTLEP}).
         This gives the existence of the tuple $\big{(}(\widetilde{M},0),\mathscr{F},\iota,i\big{)}$.
         The universality of the unfolding and 
         the uniqueness of the MFS in Proposition \ref{proposition mixed Frobenius and mixed trTLEP}
         imply the uniqueness of the tuple.
        \end{proof}
        This is a generalization of the Theorem 4.5 in \cite{H.M1}.
        \subsubsection{A proof of the unfolding theorem}\label{pf of unf}
        In this section, we give a proof of Theorem 
        \ref{unfolding theorem for mixed trTLEP-structure}.
         We define the category of unfoldings of a (filtered) trTLE-structure
         as in the case of mixed trTLEP-structure.
        \begin{proposition}\label{unfolding for filter}
        Let $\mathcal{T}_{\rm filt}$ be a filtered {\rm trTLE}-structure
        on a germ of complex manifold $(M,0)$ 
        and $(\mathcal{H},\nabla)$ the underlying trTLE-structure.
        Let $\zeta_0$ be a vector in $\mathcal{H}|_{(0,0)}$ satisfying {\rm (GC)} and {\rm (IC)}.
        Then, a universal unfolding of $\mathcal{T}_{\rm filt}$ exists and
        is characterized by the same condition as in 
        Theorem $\ref{unfolding theorem for mixed trTLEP-structure}$.
        \end{proposition} 
        To prove this proposition, let us prepare some notions.
        Let $(\nabla^{\rm r},\mathcal{V},\mathcal{C},\mathcal{U})$
        be the Frobenius type structure on $\mathcal{H}|_{\lambda=0}$
        associated to $(\mathcal{H},\nabla)$. 
        Let $V_{\mathcal{H}}$ the vector space of $\nabla^{\rm r}$-flat sections
        of $\mathcal{H}|_{\lambda=0}$.
        The dimension of $V_\mathcal{H}$ is 
        equal to the rank of $\mathcal{H}$.
        In fact, there is a canonical isomorphism  
        $(\mathcal{O}_{M,0}\otimes V_\mathcal{H},d\otimes \mathrm{id}) 
        \xrightarrow{\sim} (\mathcal{H}|_{\lambda=0},\nabla^r)$ of flat bundles
        on the germ of manifold $(M,0)$.
        We also note that for each unfolding $\widetilde{\mathcal{H}}$ of trTLE-structure, 
        the restriction map $ V_{\widetilde{\mathcal{H}}}\to V_\mathcal{H}$ is an isomorphism.        
        
        Fix a $\nabla^{\rm r}$-flat section $\zeta\in \mathcal{H}|_{\lambda=0}$.
        Since $\nabla^{\rm r}(\mathcal{C}_\bullet\zeta)=0$ as a section of 
        $\mathcal{H}|_{\lambda=0}\otimes \Omega_{M,0}^1$, 
        there is an unique section $\psi_\zeta\in \mathcal{H}|_{\lambda=0}$
        such that $\psi_\zeta(0)=0$ and 
        $\nabla^{\rm r}\psi_\zeta=\mathcal{C}_\bullet\zeta$.
        If we consider $\psi$ as a holomorphic function $\psi:(M,0)\to(V_\mathcal{H},0)$
        via the identification $\mathcal{H}|_{\lambda=0}\simeq V_\mathcal{H}\otimes \mathcal{O}_{M,0}$
        , we have $d\psi=\mathcal{C}_\bullet\zeta$.
        
        When we are given an unfolding $\widetilde{\mathcal{T}}_{\rm filt}$ of $\mathcal{T}_{\rm filt}$,
        we have a unique $\widetilde{\nabla}^{\rm r}$-flat section 
        $\widetilde{\zeta}$ such that its restriction to $(M,0)$
        equals to $\zeta$.
        Then the restriction of the holomorphic function $\psi_{\widetilde{\zeta}}$
        to $(M,0)$ equals to $\psi_\zeta$.
        \begin{lemma}\label{key lemma for unfolding}
        Let $\mathcal{T}_{\rm filt}=(\mathcal{H},\nabla,W)$ be a filtered trTLE-structure on a germ of 
        complex manifold $(M,0)$.  
        Let $\zeta\in V_\mathcal{H}$ be a $\nabla^{\rm r}$-flat section with {\rm (GC)}.
        Let $\psi$ be the holomorphic function on $(M\times \mathbb{C}^l,0)$ such that 
        $\psi|_{(M\times \{0\},0)}=\psi_\zeta$. 
        Then, there exists a unique unfolding 
        $\big{(}(M\times \mathbb{C}^l,0),\widetilde{\mathcal{T}}_\mathrm{filt},i\big{)}$
        such that $\psi=\psi_{\widetilde{\zeta}}$.
        \end{lemma}
        \begin{remark}\label{remark before key lemma}
        \begin{itemize}
        \item We can regard $V_\mathcal{H}$ as a vector space of global section of $\mathcal{H}$
              whose restriction to $\{\lambda=0\}$ is $\nabla^{\rm r}$-flat. 
              From this point of view, we have a natural isomorphism of holomorphic vector bundles 
              $\mathcal{O}_{\mathbb{P}^1_\lambda\times (M,0)}\otimes V_\mathcal{H}
               \xrightarrow{\sim} \mathcal{H}$.
        \item
          Let $\mathcal{T}_{\rm filt}={(}\mathcal{H},\nabla,W
          {)}$ 
          be filtered trTLE-structure.
        Since $(W_k(\mathcal{H}),\nabla)$ is also a trTLE-structure, we have 
        a filtration $W_k(V_\mathcal{H}):=V_{W_k(\mathcal{H})}$ on $V_{\mathcal{H}}$.
        We have a canonical isomorphism of filtered vector bundles
        \begin{align}\label{ism A}
        (\mathcal{H},W)\simeq\big{(}V_\mathcal{H}\otimes \mathcal{O}_{\mathbb{P}^1_\lambda\times (M,0)}
        ,\{W_k(V_\mathcal{H})\otimes \mathcal{O}_{\mathbb{P}^1_\lambda\times (M,0)}\}_k\big{)}.
        \end{align}
         \item
         Define an algebra 
        $\mathcal{P}_\mathcal{H}^W$ of $\mathrm{End}_\mathbb{C}(V_\mathcal{H})$
        by
        \begin{align}
        \mathcal{P}_\mathcal{H}^W
        &:=\{a\in \mathrm{End}_\mathbb{C}(V_\mathcal{H})
        \mid a(W_k)\subset W_k \text{ for all}\ k\in \mathbb{Z}\}.
        \end{align}
        Then $\mathcal{A}$ can be regarded as a subalgebra of $\mathcal{P}^W_\mathcal{H}\otimes\mathcal{O}_{M,0}$
        via the natural isomorphism $\mathcal{H}|_{\lambda=0}\simeq V_\mathcal{H}\otimes \mathcal{O}_{M,0}$.
        \end{itemize}
        \end{remark}
        \begin{proof}[Proof of Lemma $\ref{key lemma for unfolding}$]
        We may assume $l=1$.
        Put  
        $(\widetilde{M},0):=(M\times \mathbb{C},0)$, 
        $\widetilde{\mathcal{H}}:=\mathcal{O}_{\mathbb{P}^1_\lambda\times(\widetilde{M},0)}\otimes V_\mathcal{H}$,
        and 
        $\widetilde{W}_k:=\mathcal{O}_{\mathbb{P}^1_\lambda\times(\widetilde{M},0)}\otimes W_k(V_\mathcal{H})$.
        We will prove the existence and uniqueness of a meromorphic differential form $\Omega$ 
        with values in $\mathrm{End}(\widetilde{\mathcal{H}})$ 
        such that $\widetilde{\nabla}:=d+\Omega$ defines the desired
        filtered trTLE-structure
        $\widetilde{\mathcal{T}}_\mathrm{filt}:=(\widetilde{\mathcal{H}},\widetilde{\nabla},\widetilde{W})$.
 
        Take a coordinate 
        $(t,y):=(t_1,t_2,\dots,t_m,y)$ on $(M\times \mathbb{C},0)$.
        Put $\mathcal{P}(n):=\mathcal{P}^W_\mathcal{H}\otimes (\mathcal{O}_{M,0}[y]/(y)^{n+1})$
        for every non-negative integer $n$.
        Let $(\nabla^{\rm r},\mathcal{C},\mathcal{U},\mathcal{V})$ be 
        the Frobenius type structure associated to $(\mathcal{H},\nabla)$,          
        and put 
        $C_i^{(0)}:=\mathcal{C}_{\partial/\partial t_i}\ (i=1,\dots,m),
        U^{(0)}:=\mathcal{U}$, 
        and 
        $V^{(0)}:=\mathcal{V}$.
        Identifying $\mathcal{H}|_{\lambda=0}$ and 
        $\mathcal{O}_{M,0}\otimes V_\mathcal{H}$,
        we regard $C^{(0)}_i,U^{(0)}$, and $V^{(0)}$
        as a element of $\mathcal{P}(0)$.
        The meromorphic differential form $\Omega^{(0)}:=\nabla-d$
        is written as 
        \begin{align}\label{0}
         \Omega^{(0)}=\frac{1}{\lambda}\sum^m_{i=1}C_i^{(0)}\mathrm{d}t_i
                +\Big(\frac{1}{\lambda}U^{(0)}-V^{(0)}\Big{)}\frac{\mathrm{d}\lambda}{\lambda}.
        \end{align}
        \begin{claim}\label{claim1}
        For every non-negative
        integer $n$, 
        there uniquely exist 
        $(C_i^{(n)},U^{(n)},V^{(n)})
        \subset
        \mathcal{P}(n)$       
        and $C_y^{(n-1)}\in \mathcal{P}(n-1)$
        with the following properties. Here, we put $\mathcal{P}(-1):=\mathcal{P}(0)$.
        \begin{itemize}
        \item The equations 
              \begin{align}\label{n=0}
              C^{(n-1)}_y=0, C^{(n)}_i=C_i^{(0)}, U^{(n)}=U^{(0)}, V^{(n)}=V^{(0)}
              \end{align}
              are satisfied in $\mathcal{P}(0)$.
        \item The equations 
              \begin{align}
              \label{t1}
              [C_i^{(n)},C_j^{(n)}]= 
              \frac{\partial C^{(n)}_i}{\partial t_j}-\frac{\partial C^{(n)}_j}{\partial t_i}=
              [C_i^{(n)},U^{(n)}]=\frac{\partial V^{(n)}}{\partial t_i}=0 ,\\              
              \label{t2}
              \frac{\partial U^{(n)}}{\partial t_i}=[V^{(n)},C^{(n)}_i]-C^{(n)}_i
              \end{align}
              are satisfied in $\mathcal{P}(n)$.  
        \item The equations 
              \begin{align}
               \label{y1}
                [C^{(n)}_i,C^{(n-1)}_y]=
                \frac{\partial C^{(n)}_i}{\partial y}-\frac{\partial C^{(n-1)}_y}{\partial t_i}
                =[C^{(n-1)}_y,U^{(n)}]
                =\frac{\partial V^{(n)}}{\partial y}=0, \\
               \label{y2}
                \frac{\partial U^{(n)}}{\partial y}=[V^{(n-1)},C^{(n-1)}_y]-C^{(n-1)}_y
              \end{align}
              are satisfied in $\mathcal{P}(n-1)$.
              Here $\partial/\partial y:\mathcal{P}(n)\to\mathcal{P}(n-1)$ is induced from
              the differential.  
        \item The equation 
               \begin{align}\label{potential}
                 C^{(n-1)}_y(\zeta)=d\psi\Big{(}\frac{\partial}{\partial y}\Big{)}
               \end{align}
              is satisfied in $V_\mathcal{H}\otimes(\mathcal{O}_{M,0}[y]/(y)^{n})$
        \end{itemize}
        \end{claim}        
        \begin{proof}[Proof of Claim $\ref{claim1}$]
        We use an induction on $n$.
        In the case $n=0$, the flatness of $\nabla$ and the equation (\ref{0})
        imply (\ref{t1}) and (\ref{t2}). 
        Since $C_y^{(-1)}$, $\partial U^{(0)}/\partial y$, 
        $\partial C_i^{(0)}/\partial y$, and $\partial V^{(0)}/\partial y$
        are zero, 
        the equations (\ref{y1}), (\ref{y2}), and (\ref{potential}) are trivial.
        The induction step from $n$ to $n+1$ consists of the following three steps:
        \begin{description}
        \item[{\rm Step 1.}] Construction of $C_y^{(n)}\in \mathcal{P}(n)$ as a lift of $C^{(n-1)}_y$
                             so that $C_y^{(n)}$ together with $C_i^{(n)},U^{(n)}$
                             satisfies the part $[C_i^{(n)},C_y^{(n)}]=[U^{(n)},C_y^{(n)}]=0$ 
                             of (\ref{y1}) in $\mathcal{P}(n)$, 
                             and (\ref{potential}) in $V_\mathcal{H}\otimes (\mathcal{O}_{M,0}[y]/(y)^{n+1})$.
        \item[{\rm Step 2.}] Construction of $C_i^{(n+1)},U^{(n+1)},V^{(n+1)} \in \mathcal{P}(n+1)$ 
                             as a lift of $C_i^{(n)},U^{(n)},V^{(n)}$
                             such that
                             conditions (\ref{y2}) and the part 
                             ${\partial C_i^{(n+1)}}/{\partial y}-{\partial C^{(n+1)}_y}/{\partial t_i}
                             ={\partial V^{(n+1)}}/{\partial y}=0$ of (\ref{y1}) 
                             are satisfied in $\mathcal{P}(n)$.                             
        \item[{\rm Step 3.}] Check that $C_i^{(n+1)},U^{(n+1)},V^{(n+1)}$ satisfy
                             the conditions (\ref{t1}) and (\ref{t2})
                             in $\mathcal{P}(n+1)$.
        \end{description}
                Let $\mathcal{A}^{(n)}$ be a commutative subalgebra of $\mathcal{P}(n)$ 
                generated by $C_i^{(n)}$ and $U^{(n)}$.
                By (GC), the map 
                $\mathcal{A}^{(n)}\to V_\mathcal{H}\otimes(\mathcal{O}_{M,0}[y]/(y)^{n+1})$
                defined by $a\mapsto a(\zeta)$ is an isomorphism.
                Take $C_y^{(n)}$ as the inverse image of $d\psi(\partial/\partial y)$ of this isomorphism. 
                This completes Step 1.
                Step 2 is obvious.
                To prove Step 3, we use the derivation $\partial/\partial y$, 
                the equations (\ref{y1}), (\ref{y2}), and the induction hypothesis.
                For example, in $\mathcal{P}(n)$, we have
                \begin{align*}
                 &\frac{\partial}{\partial y}
                  \Big{(}
                         \frac{\partial U^{(n+1)}}{\partial t_i}
                         -\big{[}V^{(n+1)},C_i^{(n+1)}\big{]}+C_i^{(n+1)}
                  \Big{)}\\
                 &= \frac{\partial}{\partial t_i}\frac{\partial U^{(n+1)}}{\partial y}
                    -\Big{[}V^{(n)},\frac{\partial C^{(n+1)}_i}{\partial y}\Big{]}
                    +\frac{\partial C^{(n+1)}_i}{\partial y} \\
                 &=\frac{\partial}{\partial t_i}\Big{\{} \big{[}V^{(n)},C_y^{(n)}]-C_y^{(n)}\Big{\}}
                   -\Big{\{}\Big{[}V^{(n)},\frac{\partial C^{(n)}_y}{\partial t_i}\Big{]}
                            -\frac{\partial C^{(n)}_y}{\partial t_i}
                    \Big{\}} \\
                 &=0.            
                \end{align*}
            This implies that the equation (\ref{t2}) holds for $n+1$.
            The equation (\ref{t1}) is proved similarly.
                    \end{proof}
        The sequences 
        $\big{(}C_i^{(n)}\big{)}_n,\big{(}U^{(n)}\big{)}_n,\big{(}V^{(n)}\big{)}_n$, 
        and $\big{(}\mathcal{C}_y^{(n)}\big{)}_n$ 
        give formal endomorphisms
        $C_i,U,V$, and $C_y$ 
        in $\mathcal{P}^W_\mathcal{H}\otimes\mathcal{O}_{M,0}[[y]]$.
        We show that they are actually convergent.
        \begin{claim}\label{claim2}
          The endomorphisms  $C_i,U,V$, and $C_y$ 
          are in $\mathcal{P}^W_\mathcal{H}\otimes \mathcal{O}_{\widetilde{M},0}$. 
        \end{claim}
         \begin{proof}[Proof of Claim $\ref{claim2}$]
          Let $e_k\ (1\leq k\leq r=\mathrm{rank}\ \mathcal{H})$ be a 
          $\nabla^{\rm r}$-flat frame of $\widetilde{\mathcal{H}}|_{\lambda=0}$. 
          Regard the endomorphisms on $\widetilde{\mathcal{H}}|_{\lambda=0}$ 
          (or on its formal completion $\mathcal{H}|_{\lambda}\otimes\mathcal{O}_{M,0}[[y]]$)
          as $r\times r$ matrices.         
            Put $N:=r^2(m+2)$ 
            and let
            $X(t,y)\in \mathbb{C}^N\otimes \mathcal{O}_{M,0}[[y]]$ 
            be a $N$-dimensional vector valued (formal) function 
            whose entries are the entries of $C_1,\dots,C_m,U$, and $V$.
            Similarly, let $X^{(0)}\in \mathbb{C}^N\otimes \mathcal{O}_{M,0}$ 
            be a $N$-dimensional vector valued holomorphic function 
            whose entries are the entries of $C_1^{(0)},\dots,C_m^{(0)},U^{(0)}$, and $V^{(0)}$.
            The order of entries are chosen to satisfy $X(t,0)=X^{(0)}(t)$.
                       
            Let $\mathscr{A}$ be a subalgebra of 
            $\mathcal{P}^W_\mathcal{H}\otimes\mathcal{O}_{M,0}[[y]]$
            generated by $C_1,\dots,C_m$, and $U$ over $\mathcal{O}_{M,0}[[y]]$.
            By (GC), the map $\mathscr{A}\to\mathcal{H}|_{\lambda=0}\otimes\mathcal{O}_{M,0}[[y]]$
            given by $a\mapsto a{(}\widetilde{\zeta}{)}$ is an isomorphism.
           Therefore, $\mathscr{A}$ is free $\mathcal{O}_{M,0}[[y]]$-module of rank $r$.
            Take monomials $G_1,\dots,G_r$ in the endomorphisms $C_1,\dots,C_m,U$ 
            which form an $\mathcal{O}_{M,0}[[y]]$-basis of $\mathscr{A}$.
            Then, there are formal functions $g_j\in\mathcal{O}_{M,0}[[y]]$ $(1\leq j\leq r)$
            such that 
            \begin{align}\label{monomial}
             C_y=\sum_{j=1}^r g_j G_j.
            \end{align}            
            By (\ref{potential}), we have
            \begin{align}\label{potential2}
             \sum_{j=1}^r g_j G_j(\widetilde{\zeta})=\mathrm{d}\psi\Big{(}\frac{\partial}{\partial y}\Big{)}.
            \end{align}
            Since $G_j$ are monomials in $C_1,\dots,C_m$, and $U$,
            there exist $r\times r$-matrix valued functions 
            $Q_j(t,x)$ 
            such that
            the entries are in $\mathbb{C}\{t\}[x_1,x_2,\dots,x_N]$ and 
            $G_j(t,y)=Q_j\big{(}t,X(t,y)\big{)}$.
            Therefore, by the equation (\ref{potential2}) 
            and the fact that $\{G_j(\widetilde{\zeta})\mid 1\leq j\leq r\}$
            form a frame of $\mathcal{H}|_{\lambda=0}\otimes\mathcal{O}_{M,0}[[y]]$,
            there exist convergent power series $q_j(t,y,x)\in\mathbb{C}\{t,y,x\}$
            such that $g_j(t,y)=q_j\big{(}t,y,X(t,y)\big{)}$.
            Put $Q:=\sum_j q_j Q_j$.
            Then by (\ref{monomial}), we have $C_y=Q(t,y,X)$.
            
            By (\ref{y1}) and (\ref{y2}), we have the equations 
            \begin{align*}
             \frac{\partial C_i}{\partial y}=\frac{\partial C_y}{\partial t_i},\ \
              \frac{\partial V}{\partial y}=0,\ \ 
             \frac{\partial U}{\partial y}=[V,C_y]+C_y.
            \end{align*}
            Using the expression $C_y=Q(t,y,X)$ and these equations, 
            we can regard $X(t,y)$ as a formal solution of the following partial differential equation  
            \begin{align*}
              \frac{\partial X}{\partial y}(t,y)
                 &=\sum_{i=1}^mA_i(t,y,X)\frac{\partial X}{\partial t_i}(t,y) 
                   + B(t,y,X)\\
                X(t,0)&=X^{(0)}(t)
            \end{align*}
            where $A_i(t,y,x),\ (1\leq i\leq m)$ are $N\times N$ matrix whose entries 
            are in $\mathbb{C}\{t,y,x\}$ and $B_i(t,y,x)$ is $N$-dimensional 
            vector whose entries are also in $\mathbb{C}\{t,y,x\}$.
            The theorem of Cauchy-Kovalevski implies that $X(t,y)$ actually converges.
            Therefore, $C_i,U,V$ are all homomorphic and hence $C_y$ is also holomorphic by (\ref{monomial}). 
         \end{proof}
        Put  
        \begin{align}\label{Omega}
        \Omega:=
             \frac{1}{\lambda}\Big{(}\sum_{i=1}^{m}C_i\mathrm{d}t_i+C_y\mathrm{d}y\Big{)}
             +\Big{(}\frac{1}{\lambda}U-V\Big{)}\frac{\mathrm{d}\lambda}{\lambda}, 
        \end{align}
        and $\widetilde{\nabla}:=d+\Omega$.
        Then the equations (\ref{n=0}) imply that the restriction of $\widetilde{\nabla}$ to $(M,0)$ is $\nabla$.
        The equations (\ref{t1}), (\ref{t2}), (\ref{y1}), and (\ref{y2}) imply that $\widetilde{\nabla}$ is flat.
        And the equation (\ref{potential}) implies $\psi=\psi_{\widetilde{\zeta}}$.
        This proves the existence of the unfolding.
        The uniqueness in Claim \ref{claim1} implies the uniqueness of the unfolding.
        \end{proof}
         
        \begin{proof}[Proof of Proposition $\ref{unfolding for filter}$]
          By {\rm (IC)}, $\psi_\zeta:(M,0)\to (V_\mathcal{H},0)$ is closed embedding.
          Hence there exist a non-negative integer $l$ 
          and an isomorphism $\psi:(M\times \mathbb{C}^l,0)\xrightarrow{\sim}(V_\mathcal{H},0)$
          such that $\psi|_{(M\times\{0\},0)}=\psi_\zeta$.
          Applying the lemma \ref{key lemma for unfolding} for this $\psi$,
          we have an unfolding $((M\times \mathbb{C}^l,0),\mathcal{T}_{\rm filt},i)$ such that 
          $\psi_{\widetilde{\zeta}}=\psi$ and hence 
          $\widetilde{\mathcal{C}}_\bullet\widetilde{\zeta}
          :\Theta_{(M\times\mathbb{C}^l,0)}\to\widetilde{\mathcal{H}}|_{\lambda=0}$
          is an isomorphism.
          It is easy to check that this unfolding is the universal unfolding.
        \end{proof}

        The following proposition together with Proposition \ref{unfolding for filter} 
        proves Theorem \ref{unfolding theorem for mixed trTLEP-structure}.
        \begin{proposition}\label{unfolding of pairing}
         Let $\mathcal{T}=(\mathcal{H},\nabla,W,P)$ be a mixed {\rm trTLEP}-structure. 
         Assume that there is a vector $\zeta_0\in \mathcal{H}|_{(0,0)}$
         with {\rm (GC)}.
         Then, for any unfolding 
         $\big{(}(\widetilde{M},0), \widetilde{\mathcal{T}}_{\rm filt},\iota,i\big{)}$ 
         of the underlying filtered {\rm trTLE}-structure $(\mathcal{H},\nabla,W)$, 
         there exists a unique sequence of 
         graded pairings $\widetilde{P}$ on $\widetilde{\mathcal{T}}_{\rm filt}$
         such that 
         $\big{(}(\widetilde{M},0), (\widetilde{\mathcal{T}}_{\rm filt},\widetilde{P}),\iota,i\big{)}$
         is an unfolding of the mixed {\rm trTLEP}-structure $\mathcal{T}$.
        \end{proposition}
        \begin{proof}
        We may assume that $(\widetilde{M},0):=(M\times\mathbb{C},0)$. 
        Let 
        $(\widetilde{\mathcal{H}},\widetilde{\nabla})$ 
        be the underlying trTLE-structure in $\widetilde{\mathcal{T}}_{\rm filt}$
        and $\widetilde{W}$ the weight filtration.
        Fix an arbitrary integer $k$.
        The graded pairing $P_k$ uniquely extends to 
        a $\widetilde{\nabla}$-flat section $\widetilde{P}_k$ 
        on $\mathrm{Gr}^{\widetilde{W}}_k(\widetilde{\mathcal{H}})$ 
        over $\mathbb{C}^*_\lambda\times (M\times \mathbb{C},0)$.
        We need to show that it takes values on $\mathrm{Gr}_k^W(\widetilde{\mathcal{H}})$ 
        in $\lambda^{-k}\mathcal{O}_{\mathbb{P}^1_\lambda\times(\widetilde{M},0)}$.
        As in Remark \ref{remark before key lemma} and in the proof of Lemma \ref{key lemma for unfolding}
        we can normalize 
        $\widetilde{\mathcal{H}}=\mathcal{O}_{\mathbb{P}^1_\lambda\times(\widetilde{M},0)}\otimes V_\mathcal{H}$,
        $\widetilde{W}_k=\mathcal{O}_{\mathbb{P}^1_\lambda\times(\widetilde{M},0)}\otimes W_k(V_\mathcal{H})$.
        Put $\Omega=\widetilde{\nabla}-d$ and coordinate $(t,y)=(t_1,\dots,t_m,y)$ on $(M\times \mathbb{C},0)$.
        Then as in the proof of Lemma \ref{key lemma for unfolding},
        we have equation (\ref{Omega}) 
        where $C_i,C_y,U,V$ are the elements of $\mathcal{P}^W_\mathcal{H}\otimes \mathcal{O}_{\widetilde{M},0}$.
        Let $C_{i,[k]},C_{y,[k]},U_{[k]},V_{[k]}$ be the image of $C_i,C_y,U,V$
        via the morphism
        $\mathcal{P}^W_\mathcal{H}\otimes \mathcal{O}_{\widetilde{M},0}
         \to 
         \mathrm{End}\big{(}\mathrm{Gr}^W_k(V_\mathcal{H})\big{)}\otimes\mathcal{O}_{\widetilde{M},0}$.
        Define $\Omega_{[k]}$ by the following equation:
        \begin{align*}
         \Omega_{[k]}:=
             \frac{1}{\lambda}\Big{(}\sum_{i=1}^{m}C_{i,[k]}\mathrm{d}t_i+C_{y,[k]}\mathrm{d}y\Big{)}
             +\Big{(}\frac{1}{\lambda}U_{[j]}-V_{[k]}\Big{)}\frac{\mathrm{d}\lambda}{\lambda}.
        \end{align*} 
        Then $d+\Omega_{[k]}$ is the flat connection on 
        $\mathrm{Gr}^{\widetilde{W}}_k\big{(}\widetilde{\mathcal{H}}\big{)}$
        induced by $\widetilde{\nabla}$.
        
        Giving $\widetilde{P}_k$ is equivalent to give a morphism
        \begin{equation*}
         \phi_{\widetilde{P}_k}:
         \big{(}\mathrm{Gr}^{\widetilde{W}}_k\mathcal{\widetilde{H}}\big{)}
          \Big{|}_{\mathbb{C}^*_\lambda\times({\widetilde{M},0})}
          \to   
          \big{(}\mathrm{Gr}^{j_\lambda^*\widetilde{W}}_k j_\lambda^*\mathcal{\widetilde{H}}\big{)}^\vee
          \Big{|}_{\mathbb{C}^*_\lambda\times({\widetilde{M},0})}                 
        \end{equation*}
        by $\langle \phi_{\widetilde{P}_k}(u),v\rangle=\widetilde{P}_k(u,v)$
        where $\langle\bullet,\bullet \rangle $ is the natural pairing. 
        Since 
        $\widetilde{\mathcal{H}}=\mathcal{O}_{\mathbb{P}^1_\lambda\times(\widetilde{M},0)}\otimes V_\mathcal{H}$,
        and $\widetilde{W}_k=\mathcal{O}_{\mathbb{P}^1_\lambda\times(\widetilde{M},0)}\otimes W_k(V_\mathcal{H})$,
        the morphism $\phi_{\widetilde{P}_k}$ can be regarded as  
        a global section of
        $E_k
         \otimes\mathcal{O}_{\mathbb{C}^*_\lambda\times({\widetilde{M},0})}$
         where $E_k:=\mathrm{Hom}\Big{(}\mathrm{Gr}^W_k(V_{\mathcal{H}}),
        {(}\mathrm{Gr}^{j^*W}_k(V_{j^*\mathcal{H}})\big{)}^\vee\Big{)}$.         
        The flatness condition for $\widetilde{P}_k$ is equivalent to 
        \begin{align*}
         d\phi_{\widetilde{P}_k}=\Omega_{[k]}^\vee\circ \phi_{\widetilde{P}_k}+\phi_{\widetilde{P}_k}\circ j_\lambda^*\Omega_{[k]}
        \end{align*}
        which means 
        \begin{align}
        \label{flat pair 1}
         \frac{\partial}{\partial t_i}\phi_{\widetilde{P}_k}
          &=\frac{1}{\lambda}(C_{i,[k]}^\vee\circ \phi_{\widetilde{P}_k}-\phi_{\widetilde{P}_k}\circ j_\lambda^*C_{i,[k]}), \\
        \label{flat pair 2}
         \frac{\partial}{\partial y}\phi_{\widetilde{P}_k}
          &= \frac{1}{\lambda}
             (C_{y,[k]}^\vee\circ \phi_{\widetilde{P}_k}-\phi_{\widetilde{P}_k}\circ j_\lambda^*C_{y,[k]}),  \\
        \label{flat pair 3}
         \lambda\frac{\partial}{\partial \lambda}\phi_{\widetilde{P}_k}
         &=  \frac{1}{\lambda}
           (U_{[k]}^\vee\circ \phi_{\widetilde{P}_k}-\phi_{\widetilde{P}_k}\circ j_\lambda^*U_{[k]}) 
           -(V^\vee_{[k]}\circ \phi_{\widetilde{P}_k}+\phi_{\widetilde{P}_k}\circ j_\lambda^*V_{[k]} ).    
        \end{align}
        Let $\phi^{(n)}_{\widetilde{P}_k}$ be the equivalent class in 
        $E_k\otimes \big{(}\mathcal{O}_{\mathbb{C}^*_\lambda\times (M,0)}[[y]]/(y)^{n+1}\big{)}$ 
        represented by $\phi_{\widetilde{P}_k}$. 
        Then $\phi^{(0)}_{\widetilde{P}_k}$ is a global section of 
        $E_k\otimes \lambda^{-k}\mathcal{O}_{\mathbb{P}^1_\lambda\times(M,0)}$ 
        since $(\mathcal{H},W,P)$ is a mixed trTLEP-structure.
        \begin{claim}\label{P_k} 
        The pairing
        $\phi^{(n)}_{\widetilde{P}_k}$ gives a global section of 
        $E_k\otimes\big{(}\lambda^{-k}\mathcal{O}_{\mathbb{P}^1_\lambda\times (M,0)}[[y]]/(y)^{n+1}\big{)}$
        for every non-negative integer $n$.
        \end{claim}
        \begin{proof}[Proof of Claim $\ref{P_k}$]
        We use an induction on $n$.
        The case $n=0$ is explained above.
        Suppose that 
        $\phi^{(n-1)}_{\widetilde{P}_k}$ is a global section of 
        $E_k\otimes\big{(}\lambda^{-k}\mathcal{O}_{\mathbb{P}^1_\lambda\times (M,0)}[y]/(y)^{n}\big{)}$.
        Let $C_{i,[k]}^{(n)}$ be the image of $C_{[k]}$ to 
        $\mathrm{End}(\mathrm{Gr}^W_kV_\mathcal{H})\otimes
        \big{(}\mathcal{O}_{\mathbb{P}^1_\lambda\times (M,0)}[[y]]/(y)^{n+1}\big{)}$.   
        Define $C_{y,[k]}^{(n)},U_{[k]}^{(n)},V_{[k]}^{(n)}$ similarly.
        
        By (\ref{flat pair 1}), and the induction hypothesis, 
        $C_{i,[k]}^{(n-1)*}\circ P_k^{(n-1)}-P_k^{(n-1)}\circ j_\lambda^*C_{i,[k]}^{(n-1)}$
        gives a (global) section of 
        $E_k\otimes\big{(}\lambda^{-k+1}\mathcal{O}_{\mathbb{P}^1_\lambda\times (M,0)}[[y]]/(y)^{n}\big{)}$.
        Similarly, $U_{[k]}^{(n-1)*}\circ P_k^{(n-1)}-P_k^{(n-1)}\circ j_\lambda^*U_{[k]}^{(n-1)}$
        also gives a section of the same module.
        By {\rm (GC)}, 
        $C_{y,[k]}^{(n-1)}$ is an element of the algebra generated by $C_{i,[k]}^{(n-1)}$and $U_{[k]}^{(n-1)}$.
        Therefore, $C_{y,[k]}^{\vee (n-1)}\circ P_k^{(n-1)}-P_k^{(n-1)}\circ j_\lambda^*C_{y,[k]}^{(n-1)}$ 
        is a section of the same module.
        By (\ref{flat pair 2}), 
        this implies that
        $\phi^{(n)}_{\widetilde{P}_k}$ is a section of 
        $E_k\otimes\big{(}\lambda^{-k}\mathcal{O}_{\mathbb{P}^1_\lambda\times (M,0)}[[y]]/(y)^{n+1}\big{)}$.
        \end{proof}
        This claim shows that $\phi_{\widetilde{P}_k}$ is a global section of  
        $E_k\otimes\lambda^{-k}\mathcal{O}_{\mathbb{P}^1_\lambda\times (M,0)}[[y]]$.
        Since we know that $\phi_{\widetilde{P}_k}$ is analytic along $y$-direction,
        we have proved that $\phi_{\widetilde{P}_k}$ is a global section of  
        $E_k\otimes\lambda^{-k}\mathcal{O}_{\mathbb{P}^1_\lambda\times (\widetilde{M},0)}$.
       \end{proof}

      \section{Application to local B-models}
      In this section, we give an application of the construction theorem 
      (Corollary \ref{construction theorem of mixed Frobenius manifold}) to local B-models.
      
      First, in Section \ref{section VMHS}, we show that a germ of a variation of mixed Hodge structure
      with $H^2$-generation condition (\cite[Definition 5.3]{H.M1}) defines 
      a family of mixed Frobenius manifolds.
      After that, following \cite{konisi3}, 
      we recall the settings of local B-models.
      The VMHS's for local B-models are given by the relative cohomology
      group of the affine hypersurface in $(\mathbb{C}^*)^d$.
      By the work of Batyrev \cite{Baty}, Stienstra \cite{J.S}, and Konishi-Minabe \cite{konisi3}, 
      the Hodge filtrations and the weight filtrations are described by a kind of toric data.
      We recall their results in Section \ref{toric B}.
      Using these results, in Section \ref{H2 B},
      we give a sufficient condition for $H^2$-generation condition 
      in terms of the toric data
      and
      we show that the local B-model mirror to the canonical bundle of a weak Fano toric surface 
      gives rise to a mixed Frobenius manifold (Corollary \ref{construction in local B}).

      \subsection{Mixed Frobenius manifolds and variations of mixed Hodge structure}\label{section VMHS}
      We denote by $\mathscr{H}=(V_\mathbb{Q},F,W)$  
      a graded polarizable variation of mixed Hodge structure (VMHS)
      on a germ  of a complex manifold $(M,0)$.
      Here, $V_{\mathbb{Q}}$ is a $\mathbb{Q}$-local system on $(M,0)$,
      $W=(W_k\mid k\in \mathbb{Z})$ is a weight filtration on $V_{\mathbb{Q}}$, and  
      $F=(F^\ell\mid \ell\in \mathbb{Z})$ is a Hodge filtration on $K:=V_{\mathbb{Q}}\otimes\mathcal{O}_{M,0}$.      
      Recall that if $S=(S_k\mid k\in \mathbb{Z})$ is a graded polarization on 
      $\mathscr{H}$, then we have
      \begin{align}\label{graded polarization}
       S_k\big{(}\mathrm{Gr}^W_k(F^\ell), \mathrm{Gr}^W_k(F^{k-\ell+1})\big{)}=0
      \end{align}
      for any integers $k$ and $\ell$.
      \begin{definition}\label{opposite filtration}
       Fix a  graded polarization $S=(S_k\mid k\in\mathbb{Z})$ on a VMHS $\mathscr{H}=(V_\mathbb{Q},F,W)$.
       Let $\nabla:=\mathrm{id}_{V_\mathbb{Q}}\otimes \mathrm{d}$ be the flat connection on 
       $K=V_{\mathbb{Q}}\otimes\mathcal{O}_{M,0}$.
       Then an increasing filtration $U=(U_\ell \mid  \ell \in \mathbb{Z})$ 
       on $K$ is 
       called {\bf opposite filtration} if the following conditions are satisfied $:$
       \begin{enumerate}
        \item[$(a)$] $U_\ell$ is $\nabla$-flat subbundle of $K$ for each $\ell$.
        \item[$(b)$]
         For any integers $k$ and $\ell$,
          \begin{align}\label{opp}
          \mathrm{Gr}^W_k(F^\ell)\oplus\mathrm{Gr}^W_k(U_{\ell-1})&= \mathrm{Gr}^W_k(K), \\
          \label{opposite}
           S_k\big{(}\mathrm{Gr}^W_k(U_\ell),\mathrm{Gr}^W_k(U_{k-\ell+1})\big{)}&=0.
          \end{align}
       \end{enumerate}
      \end{definition}  
      \begin{remark}
       We can always construct an opposite filtration
       using the Deligne splitting. 
      \end{remark}
      Fix a VMHS $\mathscr{H}=(V_\mathbb{Q},F,W)$, a graded polarization $S$, and an opposite filtration $U$.
      Then we get a mixed trTLEP-structure as follows.
      First, let $p_\lambda:\mathbb{P}^1_\lambda\times (M,0)\to (M,0)$ be the natural projection
      and take a lattice $\mathcal{H}$ 
      of the meromorphic flat bundle 
      $\big{(}p_\lambda^*(K)(*\{0,\infty\}\times M),p^*_\lambda\nabla\big{)}$
      by 
        \begin{align}\label{F}
         \mathcal{H}|_{\mathbb{C}_\lambda\times (M,0)}
          &:=\sum_{\ell\in \mathbb{Z}} p_\lambda^*F^\ell\otimes 
             \mathcal{O}_{\mathbb{C}_\lambda\times(M,0)}(\ell\{0\}\times (M,0)),\\
          \label{U}
         \mathcal{H}|_{(\mathbb{P}^1_\lambda\setminus\{0\})\times (M,0)}
          &:=\sum_{\ell\in \mathbb{Z}} p_\lambda^*U_\ell
            \otimes\mathcal{O}_{(\mathbb{P}^1_\lambda\setminus\{0\})\times(M,0)}(-\ell\{0\}\times (M,0)). 
              \end{align}
       Then, put $\hat{W}_k:=p_\lambda^*W_k\cap \mathcal{H}$ for every integer $k$. 
      Take a pairing $P_k$ on $\mathrm{Gr}^{\hat{W}}_k(\mathcal{H})$
      by the composition of the morphism 
       \begin{align}  
         \mathrm{id}\otimes j_\lambda^*
          :\mathrm{Gr}^{\hat{W}}_k(\mathcal{H})\otimes j^*\mathrm{Gr}^{\hat{W}}_k(\mathcal{H})
                                \to
                                \mathrm{Gr}^{\hat{W}}_k(\mathcal{H})\otimes \mathrm{Gr}^{\hat{W}}_k(\mathcal{H}),    
       \end{align}
      the natural inclusion 
      $\mathrm{Gr}^{\hat{W}}_k(\mathcal{H})^{\otimes 2}
       \hookrightarrow
       \mathrm{Gr}^{p^*_\lambda W}_kp^*_\lambda K(*\{0,\infty\}\times(M,0))^{\otimes 2}$,
      and the pull back $p_\lambda^*S_k$.
      By (\ref{graded polarization}) and (\ref{opposite}),
      $P_k$ gives a morphism 
      $$P_k:\mathrm{Gr}^{\hat{W}}_k(\mathcal{H})\otimes j_\lambda^*\mathrm{Gr}^{\hat W}_k(\mathcal{H})
       \to \lambda^{-k}\mathcal{O}_{\mathbb{P}^1_\lambda\times (M,0)}.$$
      This construction is known as Rees construction.
      We get the following.
      \begin{lemma}\label{REES}
       The tuple $\mathcal{T}(\mathscr{H},S,U):=\big{(}(\mathcal{H},p^*_\lambda \nabla),(\hat{W}_k)_k,(P_k)_k\big{)}$
       defined above 
       is a mixed trTLEP-structure on $(M,0)$.       
      \end{lemma}
      \begin{proof}
       By (\ref{opp}), the adjunction map 
       $p_{\lambda*} p_\lambda^*\mathrm{Gr}^W_k(\mathcal{H})\to \mathrm{Gr}^W_k(\mathcal{H})$
       is an isomorphism for every $k$.
       Since any extension of two trivial bundles on $\mathbb{P}^1$ is trivial, 
       the adjunction map 
       $p_{\lambda*} p_\lambda^*\mathcal{H}\to\mathcal{H}$
       is also an isomorphism.
       Since $U_\ell$ is flat, the connection $p^*_\lambda\nabla$ is logarithmic along $\{\infty\}\times(M,0)$.
       The Griffith transversality implies that $p^*_\lambda\nabla$ is pole order $1$ along $\{\infty\}\times(M,0)$.
       Since $S_k$ is $(-1)^k$-symmetric and non-degenerate, 
       $P_k$ is $(-1)^k$-symmetric and non-degenerate.
      \end{proof}
     We recall the definition of $H^2$-generation condition in \cite{H.M1}.
      \begin{definition}[{\cite[Definition 5.3]{H.M1}}]\label{H2-generation condition}
       Let $\mathscr{H}:=(V_\mathbb{Q},F,W)$ be a VMHS on a germ $(M,0)$ of a complex manifold.
       Put $K:=V_\mathbb{Q}\otimes\mathcal{O}_{M,0}$, $\nabla:=\mathrm{id}_{V_\mathbb{Q}}\otimes d$,
       and $w:=\max\{l\in\mathbb{Z}\mid F^l\neq0\}$.
       Let $\theta:=\mathrm{Gr}_F(\nabla):\mathrm{Gr}_FK\to \mathrm{Gr}_FK\otimes\Omega^1_{M,0}$
       be the induced Higgs field.
       The {\bf $H^2$-generation condition} for $\mathscr{H}$ 
       is the following.
       \begin{enumerate}
         \item[\rm (i)] The rank of $F^w$ is $1$,
                    and the rank of $\mathrm{Gr}^{w-1}_F(K)$ 
                    is equal to the dimension of $(M,0)$,
         \item[\rm (ii)] The map $\mathrm{Sym}\ \Theta_{M,0}\otimes F^w \to \mathrm{Gr}_F K$ 
                     induced by $\theta$ is surjective.
         \end{enumerate}

      \end{definition}
      Standard discussion on Rees construction shows the following.
      \begin{lemma}
       Let $\mathscr{H}=(V_\mathbb{Q},F,W)$ 
       be a VMHS on a germ of complex manifold $(M,0)$.
        Take the integer $w$ as above and non-zero vector $\zeta_0\in F^w|_0$.
         Fix a graded polarization $S$ and
         an opposite filtration $U$. 
       \begin{enumerate}
       \item[$(a)$]
        The vector $\zeta_0$        
        satisfies ${\rm(EC)}_{2w}$ with respect to $\mathcal{T}(\mathscr{H},S,U)$.
       \item[$(b)$]
        Assume moreover that the rank of $F^w$ is $1$.
        Then, the vector $\zeta_0$        
        satisfies {\rm(GC), (IC)}
        if and only if $\mathscr{H}$ satisfies $H^2$-generation condition.
       \end{enumerate}
      \end{lemma}
      \begin{proof}
      Let $(\mathcal{H},p^*_\lambda\nabla)$ be the underlying trTLE-structure in $\mathcal{T}(\mathscr{H},S,U)$
      (defined by (\ref{F}) and (\ref{U})).
      Let $(\nabla^r,\mathcal{C},\mathcal{U},\mathcal{V})$ be the associated Frobenius type structure 
      (Lemma \ref{Frob type trTLE}).
      Then, using the decomposition $K= \bigoplus_\ell F^\ell\cap U_\ell$, 
      we have $\mathcal{V}=\ell\cdot\mathrm{id}$ on $F^\ell\cap U_\ell$.
      Since $\zeta_0$ is in $F^w\cap U_w$, this proves $(a)$.
      The Higgs field $\mathcal{C}$ corresponds to $\theta$ via the natural isomorphism 
      $\mathcal{H}|_{\lambda=0}\simeq \mathrm{Gr}^FK$.
      We also remark that $\mathcal{U}=0$.
      Hence the condition {\rm (ii)} in Definition \ref{H2-generation condition} is equivalent to {\rm (GC)}.
      If we assume the condition {\rm (ii)}, the morphism $\Theta_{M,0}\to \mathrm{Gr}^F_{w-1}K$ is surjective. 
      Then the morphism is injective (this is equivalent to (IC)) if and only if the rank of $\mathrm{Gr}^F_{w-1}K$
      is equal to the dimension of $M$. This proves (b).   
      \end{proof}
      Hence, combining Corollary \ref{construction theorem of mixed Frobenius manifold},
      we have the following.
      \begin{corollary}\label{H2-generated mixed Frobenius}
      
      Let $\mathscr{H}=(\mathbb{V}_\mathbb{Q},F,W)$ be a VMHS 
         on a germ $(M,0)$ of a complex manifold  with $H^2$-generation condition.
         Take the integer $w$ as above and non-zero vector $\zeta_0\in F^w|_0$.
         Fix a graded polarization $S$ and
         an opposite filtration $U$.
         Then there exists a tuple $(\mathscr{F},\iota, i)$ with following conditions up to 
         isomorphisms.
         \begin{enumerate}
         \item[$1.$] $\mathscr{F}$ is a MFS of charge $2w$ on a germ  of a complex manifold $(\widetilde{M},0)$.
         \item[$2.$] $\iota: (M,0)\hookrightarrow (\widetilde{M},0)$ is a closed embedding.
         \item[$3.$] $i:\mathcal{T}(\mathscr{H},U,S)\xrightarrow{\sim} \mathcal{T}(\mathscr{F})$ is
                     an isomorphism of mixed {\rm trTLEP}-structure with 
                     $i|_{(0,0)}(\zeta_0)=e|_0$ where $e$ is the unit vector field of $\mathscr{F}$.\qed 
         \end{enumerate}
        \end{corollary}
       We give an example of VMHS which satisfies $H^2$-generation condition.
         \begin{definition}
         Let $Y$ be a projective complex manifold and put $d:=\dim Y$. 
         Let $D_i\ (i=0,1)$ be 
         hypersurfaces in $Y$ 
         such that $D_0$ is smooth and  
         $D:=D_0\cup D_1$ is normal crossing.
         The triple $(Y,D_0,D_1)$ is called an {\bf open Calabi-Yau manifold with a divisor} 
         if $\Omega_Y^{d}(D_1)$ is trivial.
         \end{definition}
         
         \begin{remark}\label{deg}
         Take an open Calabi-Yau manifold with divisor $(Y,D_0,D_1)$.
         Put $d:=\dim Y$.
         Let $F$ be the Hodge filtration on $H^d(Y\setminus D_1,D_0\setminus D_1)$.
         \begin{enumerate}
           \item[\rm 1.] By the degeneration of Hodge-to-de Rham spectral sequence, 
                         we have the following.
                         \begin{align}\label{degeneration of Hodge-to-de Rham}
                         \mathrm{Gr}_F^pH^d(Y\setminus D_1,D_0\setminus D_1)
                         \simeq H^{d-p}\big{(}Y,\Omega^p(\log D)(-D_0)\big{)}.
                         \end{align}
           \item[\rm 2.] Since $\Omega_Y^{d}(D_1)$ is trivial, the dimension of   
                         $\mathrm{Gr}^d_FH^d(X\setminus D_1,D_0\setminus D_1)
                         \simeq H^{0}(Y,\Omega^d(D_1))$ is $1$.
         \end{enumerate}
         \end{remark}
         \begin{definition}\label{generation condition for CY with div}
           We say that an open Calabi-Yau manifold with a divisor $(X,D_0,D_1)$ satisfies 
           {\bf $H^2$-generation condition} if the natural morphism
           \begin{equation}\label{H^2 for CY}
            \mathrm{Sym}\Big{(}H^1\big{(}Y,\Theta_Y(\mathrm{log}D)\big{)}\Big{)}
            \otimes H^0\big{(}Y,\Omega_Y^d(D_1)\big{)}
            \to \mathrm{Gr}_FH^d(Y\setminus D_1,D_0\setminus D_1)
           \end{equation}
           is surjective.
         \end{definition}
         \begin{remark}\label{1-dim case}
          If $d=1$, then $H^1(Y,\Theta_Y(-D))$ is isomorphic to 
          $H^1(Y,\mathcal{O}(-D_0))\simeq \mathrm{Gr}^1_F$ and hence 
          $H^2$-generation condition is automatically satisfied.
         \end{remark}
         
         We then consider a complete family of open Calabi-Yau manifold with a divisor.
         That is, we consider a smooth projective morphism $\pi:(\mathcal{Y},Y)\to (M,0)$ and 
         divisors $(\mathcal{D}_i,D_i)$ $(i=0,1)$ with the following properties.
         \begin{itemize}
           \item $\mathcal{D}_0$ is smooth  and $\mathcal{D}:=\mathcal{D}_0\cup \mathcal{D}_1$
                 is normal crossing in $\mathcal{Y}$.
           \item $\Omega^d_{\mathcal{Y}/M}(\mathcal{D}_1)$
                 is isomorphic to 
                 $\mathcal{O}_\mathcal{Y}$ 
                 where $d=\dim \mathcal{Y}-\dim M$.
           \item The Kodaira-Spencer morphism  
                 $\rho:\Theta_{M,0}\to R^1\pi_*\Theta_{\mathcal{Y}/M}(\log \mathcal{D})$ is an isomorphism.
         \end{itemize}
         Let $j^1:\mathcal{Y}\setminus \mathcal{D} \hookrightarrow \mathcal{Y}\setminus \mathcal{D}_0$ and 
         $j^2:\mathcal{Y}\setminus \mathcal{D}_0\hookrightarrow \mathcal{Y}\setminus$ be the inclusions.
         Then $R^d\pi_*j^2_!j^1_*\mathbb{Q}_{Y\setminus D}$ gives 
         a VMHS on $(M,0)$ which we denote by $\mathscr{H}$.     
         \begin{lemma}
          The VMHS $\mathscr{H}$ satisfies $H^2$-generation condition 
          in the sense of Definition $\ref{H2-generation condition}$
          if the open Calabi-Yau manifold with a divisor $(Y,D_0,D_1)$ 
          satisfies $H^2$-generation condition 
          in the sense of Definition $\ref{generation condition for CY with div}$.
         \end{lemma}
         \begin{proof}
         By Remark \ref{deg} and since the Kodaira-Spencer map $\rho$ is an isomorphism, 
         the condition (i) in Definition \ref{H2-generation condition} is satisfied.
         The natural pairing
          $$\Theta_{\mathcal{Y}/\mathcal{M}}(\log \mathcal{D})\otimes 
          \Omega^p_{\mathcal{Y}/\mathcal{M}}(\log \mathcal{D}_1)(-\mathcal{D}_0) 
          \to \Omega^{p-1}_{\mathcal{Y}/\mathcal{M}}(\log \mathcal{D}_1)(-\mathcal{D}_0) $$ 
         induces the morphism 
          $$R^1\pi_*\Theta_{\mathcal{Y}/\mathcal{M}}(\log \mathcal{D})
          \otimes R^q\pi_*\Omega^p_{\mathcal{Y}/\mathcal{M}}(\log \mathcal{D}_1)(-\mathcal{D}_0)
          \to  R^{q+1}\pi_*\Omega^{p-1}_{\mathcal{Y}/\mathcal{M}}(\log \mathcal{D}_1)(-\mathcal{D}_0).$$
         Using the Kodaira-Spencer morphism $\rho$ and (\ref{degeneration of Hodge-to-de Rham}), 
         this corresponds to 
         $$\mathrm{Gr}_F(\nabla): \Theta_{M,0}\otimes F^d\to \mathrm{Gr}^d_FH^d(Y\setminus D_1,D_0)$$
         at $0\in (M,0)$. Therefore the surjectivity of (\ref{H^2 for CY}) implies 
         the condition (ii) in Definition \ref{H2-generation condition}.         
         \end{proof}
         \begin{example}{\rm 
         Put $Y:=\mathbb{P}^1$, $D_1:=\{0,\infty\}$, and $D_0:=\{1,z_1,\dots z_m\}$ 
         where $z_i\neq 0,1,\infty\ (i=1,2,\dots m)$ 
         and $z_i\neq  z_j\ (i\neq j)$.
         Then $(Y,D_0,D_1)$ is a open Calabi-Yau manifold with a divisor.
         As mentioned in Remark \ref{1-dim case}, 
         this satisfies the $H^2$-generation condition and
         hence the complete family of $(Y,D_0,D_1)$ gives rise to a mixed Frobenius manifold.}
         \end{example}
        
                        
        \subsection{Combinatorial description of VMHS for local B-models}\label{toric B} 
        \subsubsection{Settings for local B-models}\label{set B}        
        Let $N$ be a finitely generated free abelian group and $d$ the rank of $N$.
        Let $N^{\vee}$ be the dual lattice of $N$ and put
        $\overline{N^{\vee}}:=N^{\vee}\oplus \mathbb{Z}$.
        Consider the group ring 
        $\mathbb{C}[\overline{N^{\vee}}]=\mathbb{C}[t_0,t^{-1}_0]\otimes \mathbb{C}[N^{\vee}]$
        as a graded ring by $\deg (t_0^kt^m):=k\ (k\in \mathbb{Z}, m\in N^{\vee})$.
        For an integral polyhedron $\Delta\subset N^{\vee}_{\mathbb{R}}:=N^{\vee}\otimes \mathbb{R}$, 
        let $\sigma_\Delta$ be the cone in $\overline{N^{\vee}_\mathbb{R}}$ generated by $\{1\}\times \Delta$.
        This defines a graded subring 
        $S_\Delta:=\mathbb{C}[\sigma_\Delta\cap N^{\vee}]$ in $\mathbb{C}[\overline{N^{\vee}}]$.
        $\mathbb{P}_\Delta:=\mathrm{Proj}\ S_\Delta$ is a toric variety which contains an algebraic torus 
        $T_N:=\mathrm{Spec}\ \mathbb{C}[N^{\vee}]$ as an open dense subset.
        We also note that $D_\Delta:=\mathbb{P}_\Delta\setminus T_N$ is a hypersurface.
        
        Recall that the Newton polygon of a Laurent polynomial
        $f=\sum_{m\in N^{\vee}}a_mt^m\in\mathbb{C}[N^{\vee}]$ 
        is the convex hull of the subset 
        $\{m\in N^{\vee}\mid a_m\neq 0\}$ in $N^{\vee}_\mathbb{R}$.
        Put $A:=A(\Delta):=\Delta\cap N^{\vee}$ and let
        $\mathbb{L}(\Delta)$ be the set of functions 
        whose newton polygon is contained in $\Delta$. 
        Then $\mathbb{L}(\Delta)$ is naturally identified with 
        $\mathbb{C}^A$.
        \begin{definition}
        Let $\Delta'$ be a face of $\Delta$. 
          For $f:=\sum_{m\in A}a_mt^m\in \mathbb{L}(\Delta)$, we define the function $f^{\Delta'}$ by
          $f^{\Delta'}:=\sum_{m\in \Delta'\cap N^{\vee}}a_mt^m$.         
        \end{definition}
        For a basis $u_1,u_2,\dots,u_d$ of $N$, let $\theta_1,\theta_2,\dots,\theta_d$ be the corresponding 
        vector field on $T_N$. 
        Each $\theta_i$ defines a differential operator 
        on $\mathbb{C}[N^{\vee}]$ by $\theta_i(t^m)=\langle u_i, m\rangle t^m$.
        \begin{definition}[{\cite[Definition 3.1]{konisi3}}]
        A Laurent polynomial $f\in\mathbb{C}[N^{\vee}]$ is called {\bf $\Delta$-regular}
        if the following conditions are satisfied.
         \begin{enumerate}
          \item[$1.$] The Newton polygon of $f$ is $\Delta$.
          \item[$2.$] For each face $\Delta'$ of $\Delta$, there is no point in $T_N$ such that 
            \begin{equation}
                 f^{\Delta'}=\theta_1(f^{\Delta'})=\cdots=\theta_d(f^{\Delta'})=0.
            \end{equation}
         \end{enumerate}
         Let $\mathbb{L}_{\rm reg}:=\mathbb{L}_{\rm reg}(\Delta)$ be the set of $\Delta$-regular Laurent polynomials.
        \end{definition}
        \subsubsection{Mixed Hodge structure}
          For $f\in \mathbb{L}(\Delta)$, we define the differential operators 
          $\mathcal{L}_{f}^i\ (0\leq i\leq d)$ on $S_\Delta$ by 
           \begin{equation}\label{D_f}
              \mathcal{L}_{f}^0:=t_0\partial_{t_0}+t_0f,\ 
              \mathcal{L}_{f}^i:=\theta_i+t_0\theta_if,
              \ (i=1,2,\dots,d).
           \end{equation}
        
        \begin{definition}
           We define the vector space $\mathcal{R}_f$ by
           \begin{equation}
             \mathcal{R}_f:=S_\Delta/\sum_{i=0}^d\mathcal{L}_{f}^iS_\Delta.
           \end{equation}
        \end{definition}
        We define the decreasing filtration $\mathcal{E}$ on $S_\Delta$ by 
        $\mathcal{E}^\ell:=\bigoplus_{\ell\leq k}S_\Delta^k$. 
        We denote the induced filtration on $\mathcal{R}_f$ by the same letter.
        Denote by $\sigma_\Delta(\ell)$ the set of the co-dimension $\ell$ faces of $\sigma_\Delta$.
        Put $|\sigma_\Delta(\ell)|:=\bigcup_{\tau\in\sigma_\Delta(\ell)}\tau$ and
        $I(\ell):=(\sigma_\Delta\setminus |\sigma_\Delta(\ell)|)\cap \overline{N^\vee}$.
        Then an increasing filtration 
        \begin{equation}
           \mathcal{I}_\ell:=\bigoplus_{(k,m)\in I(\ell)}
           \mathbb{C} t_0^k t^m        
        \end{equation}
        on $S_\Delta$ is defined.
        We also denote the induced filtration on $\mathcal{R}_f$ by the same letter.

        For $f\in \mathbb{L}_{\rm reg}(\Delta)$, 
        put $V_f:=f^{-1}(0)$. 
        In \cite{konisi3}, Konishi and Minabe constructed an isomorphism 
         \begin{equation}\label{rho}
                             \rho:\mathcal{R}_f\xrightarrow{\sim}H^d(T_N,V_f)
          \end{equation}
         with the following properties.
           \begin{enumerate}                          
             \item[$(a)$] If ${F}=(F^\ell\mid \ell\in \mathbb{Z})$ is the Hodge filtration on $H^d(T_N,V_f)$,
                           then $\rho(\mathcal{E}^{i-d})=F^i$ for $0 \leq i\leq d$.             
             \item[$(b)$] If $W=(W_k\mid k\in\mathbb{Z})$ is the weight filtration on $H^d(T_N,V_f)$, then 
                            \begin{align*}
                             \rho(\mathcal{I}_i)=W_{d-2+i},\ (0<i\leq d-1),\ 
                             \rho(\mathcal{I}_{d+1})=W_{2d-2}=W_{2d-1},\
                             H^d(T_N,V_f)=W_{2d}.
                            \end{align*}
           \end{enumerate}
                            
         \subsubsection{Gauss-Manin connection}
          Let $\mathcal{O}_{\mathbb{L}(\Delta)}$ 
          be the sheaf of algebraic functions on $\mathbb{L}(\Delta)$.
          Since $\mathbb{L}(\Delta)=\mathbb{C}^A$, 
          we have $\mathbb{C}[(a_m)_{m\in A}]=H^0(\mathbb{L}(\Delta),\mathcal{O}_{\mathbb{L}(\Delta)})$.
           For $i=0,\dots, d$, let $\mathcal{L}^i$ be the differential operator 
           on $S_\Delta\otimes \mathcal{O}_{\mathbb{L}(\Delta)}$
           given by 
           \begin{align}
           \mathcal{L}^0:=t_0\partial_0+t_0\sum_{m\in A}a_mt^m, \ 
           \mathcal{L}^i:=\theta_i+\sum_{m\in A} a_m\langle u_i,m \rangle t^m\ (1\leq i\leq d).
           \end{align}
          Remark that $\mathcal{L}^i=\mathcal{L}^i_f$ at $f\in \mathbb{L}(\Delta)$.
          We define the $\mathcal{O}_{\mathbb{L}(\Delta)}$-module $\mathcal{R}$ by 
           \begin{equation}
             \mathcal{R}:=
              S_\Delta \otimes \mathcal{O}_{\mathbb{L}(\Delta)}
              \Big{/} \sum_{i=1}^d\mathcal{L}^i(S_\Delta\otimes \mathcal{O}_{\mathbb{L}(\Delta)}).           
           \end{equation}   
          The restriction of $\mathcal{R}$ to $\mathbb{L}_{\rm reg}(\Delta)$ defines an algebraic 
          vector bundle, which we denote by the same letter.
          We note that the fiber of $\mathcal{R}$ at $f\in \mathbb{L}_{\rm reg}(\Delta)$ is
          $\mathcal{R}_f$.
          
           We define differential operators $\mathcal{D}_{a_m}\ (m\in A)$ 
          on $S_\Delta\otimes \mathcal{O}_{\mathbb{L}(\Delta)}$ by
           \begin{equation}
            \mathcal{D}_{a_m}:=\frac{\partial}{\partial a_m}+t_0t^m.
           \end{equation}           
         Let $\nabla$ be the connection on $\mathcal{R}$ defined by $\nabla_{\partial_{a_m}}:=\mathcal{D}_{a_m}$.
          
          Put $\mathcal{X}:=\mathbb{P}_\Delta\times \mathbb{L}_{\rm reg}(\Delta),
          \ \widetilde{\mathcal{M}}:=\mathbb{L}_{\rm reg}(\Delta)$ 
          and let $\pi:\mathcal{X}\to \widetilde{\mathcal{M}}$ be the projection.
          Define the divisors $D_0, D_1$ and $D$ by
          ${D}_0:=\{(p,f)\in \mathcal{X}\mid p\in V_{f}\}$, 
          ${D}_1:=D_\Delta\times \mathbb{L}_{\rm reg}(\Delta)$,
          and 
          ${D}:={D}_0\cup {D}_1$.
          Let
          $j^1:\mathcal{X}\setminus D_0\hookrightarrow \mathcal{X}$, 
          $j^2:\mathcal{X}\setminus D\hookrightarrow \mathcal{X}\setminus D_0$
          be the inclusions.
          The stalk of the sheaf 
          $R^d\pi_*j^1_!j^2_*\mathbb{C}_{\mathcal{X}\setminus D}$
          at $f\in \widetilde{\mathcal{M}}$
          is $H^d(T_N,V_f)$.
          \begin{lemma}[{\cite[Lemma 4.1]{konisi3},\cite[Section 6]{J.S}}]\label{loc sys}
            The isomorphism $(\ref{rho})$ gives an isomorphism between 
            the local system of flat section 
            of the analytic flat bundle 
            $(\mathcal{R},\nabla)^{\rm an}$
            and 
            $R^d\pi_*j^1_!j^2_*\mathbb{C}_{\mathcal{X}\setminus D}$.\qed
          \end{lemma}
          
          \subsubsection{The moduli space}
          We recall the definition of the moduli space $\mathcal{M}(\Delta)$ 
          of the affine hypersurfaces of $T_N$.
          Define the action of $T_N$ on $\mathbb{L}(\Delta)$ by
          $(sf)(t):=f(st)$ where $f\in \mathbb{L}(\Delta)$ and $s,t\in T_N$.
          We put $\mathbb{C}[a_m]:=\mathbb{C}[(a_m)_{m\in A}]$.
          We regard the invariant ring $\mathbb{C}[a_m]^{T_N}$ as a graded ring using 
          the natural grading on $\mathbb{C}[a_m]$.
          \begin{definition}[{\cite[Definition 10.4]{Baty}}]\label{moduli of hypersurface}
             We define the moduli space of affine hypersurface in $T_N$ by
             $\mathcal{M}(\Delta):=\mathrm{Proj}(\mathbb{C}[a_m]^{T_N})$.
          \end{definition}
          If we put $\mathbb{P}(A):=\mathrm{Proj}(\mathbb{C}[a_m])$, 
          $\mathcal{M}(\Delta)$ is a GIT-quotient of $\mathbb{P}(A)$ by the action of $T_N$.
          The stability condition of this GIT-quotient is defined as follows.
          \begin{definition}
            For a point $x\in \mathbb{P}(A)$, 
            take a non-zero vector $v\in\mathbb{C}^A$
            which represents $x$.
            The point $x$ is stable if the orbit $T_N\cdot v$ is closed and of $d$-dimensions.
          \end{definition}
          Put $R_f^i:=\mathrm{Gr}^{-i}_\mathcal{E}\mathcal{R}_f$,
          $R_f:=\bigoplus_iR^i_f$.
          By the property $(a)$ of the isomorphism (\ref{rho}),
          we have the isomorphism 
           \begin{equation}\label{Higgs field}
             R^i_f\simeq \mathrm{Gr}^{d-i}_{F}H^d(T_N,V_f),\ (0\leq i\leq d).
           \end{equation}
          Let $J_f$ be a homogeneous ideal of $S_\Delta$ generated  by $t_0f,t_0\theta_1f,\dots,t_0\theta_df$.
          Then we naturally have the isomorphism $R_f\simeq S_\Delta/J_f$ of graded rings.
          \begin{proposition}[{\cite[Proposition 11.2, Corollary 11.3]{Baty}}]\label{GM on moduli}
          Consider the action of the torus $T_{\overline{N}}$ on $\mathbb{L}(\Delta)$
          \begin{equation}
           T_{\overline{N}}\times \mathbb{L}(\Delta)\to\mathbb{L}(\Delta):(t_0,t)\times f(s)\mapsto t_0f(ts).
          \end{equation}
          If we identify $L(\Delta)$ and $S^1_\Delta$ by $f\mapsto t_0f$, 
          the tangent space to the orbit $T_{\overline{N}}f$ is isomorphic to the homogeneous component $J^1_f$.         
          Moreover, if we assume that $f\in\mathbb{L}_{\rm reg}(\Delta)$, and 
          the corresponding class
          $[f]\in \mathcal{M}(\Delta)$ is a smooth stable point.
          Then, the tangent space $\Theta_{\mathcal{M}(\Delta),[f]}$ is naturally
          isomorphic to $R^1_f=S^1_\Delta/J^1_f$. \qed
          \end{proposition}

          \begin{proposition}\label{toric higgs}
            Let $f$ be a $\Delta$-regular function in $\mathbb{L}(\Delta)$ such that 
            corresponding class $[f]$ in $\mathcal{M}(\Delta)$ is smooth stable point.
            Then under the isomorphism $(\ref{Higgs field})$ and 
            the isomorphism $\Theta_{\mathcal{M}(\Delta),[f]}\simeq R^1_f$ in Proposition $\ref{GM on moduli}$,
            the Higgs field 
            \begin{equation*}
             \mathrm{Gr}_{F}(\nabla)_{[f]}
             :\Theta_{\mathcal{M}(\Delta),[f]}\otimes \mathrm{Gr}_{F}H^d(T_N,V_f)
             \to \mathrm{Gr}_{F}H^d(T_N,V_f)
            \end{equation*}
            corresponds to the multiplication 
            $$R^1_f\otimes R_f\to R_f.$$
          \end{proposition}
          \begin{proof}
          By Lemma \ref{loc sys}, 
          the Gauss-Manin connection corresponds to the connection 
          \begin{equation}
           \nabla_{\partial_{a_m}}=\frac{\partial}{\partial a_m}+t^0t^m\ \ (m\in A)
          \end{equation}
          on $\mathcal{R}$ over $\mathbb{L}(\Delta)$.
          Since the filtration $\mathcal{E}$ is determined by the degree of $t_0$, 
          under the identifications  $\mathbb{L}(\Delta)\simeq S_\Delta^1$
          and (\ref{Higgs field}),  
          the Higgs field
          $$\mathrm{Gr}_F(\nabla):
           \Theta_{\mathbb{L}(\Delta),f}\otimes \mathrm{Gr}_{F}H^d(T_N,V_f)\to \mathrm{Gr}_{F}H^d(T_N,V_f)$$
          corresponds to the multiplication
          $$S_\Delta^1\otimes R_f\to R_f.$$
          By Proposition \ref{GM on moduli}, 
          this implies the conclusion.
          \end{proof}
      \subsection{$H^2$-generation condition and MFS for local B-models}\label{H2 B}
          Let $[f_0]\in \mathcal{M}(\Delta)$ be a smooth stable point and assume that $f_0$ is $\Delta$-regular.
          And let $\mathscr{H}_\Delta$ be the variation of mixed Hodge structure on the germ of complex 
          manifold $(\mathcal{M}(\Delta),[f_0])$ defined by $H^d(T_N,V_{f}), (f\in \mathcal{M}(\Delta))$.
          By proposition \ref{toric higgs}, $\mathscr{H}_\Delta$ satisfies the $H^2$-generation condition 
          if and only if $R_f$ is generated by $R^1_f$.
          In this section, we consider the following condition: 
          $S^1_\Delta$ generates $S_\Delta$.
          If this condition is satisfied, then $\mathscr{H}_\Delta$ satisfies the $H^2$-generation condition
          and hence gives rise to a mixed Frobenius manifold. 
          
          \begin{definition}[{\cite[Definition 12.3]{Baty}}]
           A polyhedron $\Delta\subset N^{\vee}_\mathbb{R}$ is called {\bf reflexive}
           if it satisfies the following conditions.
             \begin{enumerate}
              \item[\rm 1.] $\Delta$ contains $0\in N^{\vee}$.
              \item[\rm 2.] For any codimension $1$ face $\Delta'$, 
                            there exists a primitive element $u\in N$ 
                            such that 
                             \begin{equation*}
                               \Delta'=\{m\in N^\vee_\mathbb{R}\mid \langle m,u\rangle=-1\}.
                             \end{equation*}
             \end{enumerate}
          \end{definition}
          In the following, we assume that $\Delta$ is reflexive.
          \begin{remark}
           Reflexive polyhedron $\Delta$ has following properties $(${\rm \cite[\rm Theorem 12.2]{Baty}}$)$.
            \begin{enumerate}
             \item[\rm 1.] Its dual polyhedron 
                           $\Delta^*:=\{u\in N_\mathbb{R}\mid \langle\Delta, u\rangle\geq -1\}$
                           is also reflexive. 
             \item[\rm 2.] $D_\Delta=\mathbb{P}_\Delta\setminus T_N$ is an 
                           anti-canonical divisor of $\mathbb{P}_\Delta$.
             \item[\rm 3.] $\mathbb{P}_\Delta$ is a Fano variety with Gorenstein singularities.
            \end{enumerate}
          \end{remark}
          
          \begin{lemma}\label{2-dim lemma}
            Let $\Delta\subset N^\vee_\mathbb{R}$ be a $2$-dimensional reflexive polyhedron.
            Then, $S_\Delta$ is generated by $S^1_\Delta$.
          \end{lemma}
          \begin{proof}
            Fix an isomorphism $N^\vee\simeq \mathbb{Z}^2$.
            Label the elements of $\Delta\cap N^\vee\setminus \{0\}$ 
            anti-clockwise with $\{m_1,m_2,\dots m_l\}$. 
            For each $i$, let $\tau_i$ be the cone 
            generated by $m_i$ and $m_{i+1}$. (Here, we put $m_{\ell+1}:=m_1$).
            The cones $\{\tau_i\}_i$ define a complete fan, which we denote by $\Sigma(\Delta^*)$.
            It is known that the toric manifold corresponding to $\Sigma(\Delta^*)$
            is smooth and weak Fano.
            Therefore, the pair $\{m_i,m_{i+1}\}$ is a basis of $N^\vee$ for every $i$.
            Put $\sigma_i:=\mathrm{Cone}((1,0),(1,m_i),(1,m_{i+1}))\subset \mathbb{R}\times N^\vee_\mathbb{R}$
            and $S_i:=\mathrm{Spec}(\mathbb{C}[\sigma_i\cap (\mathbb{Z}\times N^\vee)])$.
            Since $\{m_i,m_{i+1}\}$ is a basis of $N^\vee$, each $S_i$ is generated by $S^1_i$.
            The equation $S_\Delta=\sum_iS_i$ shows the lemma. 
          \end{proof}
          For higher dimensional case, we consider following condition.
          \begin{definition}[{\cite[Definition 12.5, Remark 12.6]{Baty}}]
            Let $\Delta$ be a reflexive polyhedron and $\Delta^*$ its dual.
            Then $\Delta$ is called {\bf Fano polyhedron} if $\mathbb{P}_\Delta$ is smooth 
            Fano variety.
          \end{definition}
          
          \begin{lemma}[{\cite[Lemma 12.9]{Baty}}]\label{batyrev lemma}
            If $\Delta$ is Fano polyhedron, then $S_\Delta$ is generated by
            $S^1_\Delta$.\qed
          \end{lemma}
          
          Now, we assume that $\Delta$ is 2-dimensional or Fano polyhedron.
          Fix a Laurent polynomial $f_0\in \mathbb{L}_{\rm reg}(\Delta)$ such that 
          corresponding $[f_0]\in \mathcal{M}(\Delta)$ is smooth stable point.
          \begin{corollary}\label{construction in local B} 
          Let $\mathscr{H}_\Delta$ be the variation of mixed Hodge structure 
          on $(\mathcal{M}(\Delta),[f_0])$ defined by $H^d(T_N,V_f),\ ([f]\in\mathcal{M}(\Delta))$.
          Fix a graded polarization $S$ and an opposite filtration $U$. 
          Fix a generator $\zeta_0$ of $\mathrm{Gr}_F^d(H^d(T_N,V_{f_0}))$.
          There  exists the tuple $\big{(}(\widetilde{M},0)\mathscr{F},\iota,i\big{)}$ with 
          following properties uniquely up to isomorphisms.
              \begin{itemize}
                \item $\mathscr{F}=({\bm \nabla},\circ,E,e,\mathcal{I},g)$ is a MFS of charge $2d$
                      on a germ $(\widetilde{M},0)$ of a complex manifold.
                \item $\iota:(\mathcal{M}(\Delta),[f_0])\hookrightarrow (\widetilde{M},0)$ is a closed embedding.
                \item $i:\mathcal{T}(\mathscr{H}_\Delta,S,U)
                         \to \iota^*\mathcal{T}(\mathscr{F})$ 
                      is an isomorphism of mixed trTLEP-structure with
                      $i|_{(0,0)}(\zeta_0)=e|_0$.
              \end{itemize}
          \end{corollary}
          \begin{proof}
          By Lemma \ref{2-dim lemma} and Lemma \ref{batyrev lemma}, 
          $\zeta_0$ generates
          $\mathrm{Gr}_{F}H^d(T_N,V_{f_0})$ over $\Theta_{\mathcal{M}(\Delta), [f_0]}$.
          Hence, by Corollary \ref{H2-generated mixed Frobenius}, 
          we have the conclusion.
          \end{proof}          
            \section{Application to local A-models}
          In this section, we give an application of the construction theorem 
      (Corollary \ref{construction theorem of mixed Frobenius manifold}) to local A-models.
                 \subsection{Limit mixed trTLEP-structure}\label{limit G}
      
       \subsubsection{Mixed trTLEP-structure defined by a nilpotent endomorphism}
       Let $(\mathcal{H},\nabla,P)$ be a trTLEP($0$)-structure on a complex manifold $M$. 
       Let $p_\lambda:\mathbb{P}^1_\lambda\times M\to M$ be the projection.
       Assume that there is a nilpotent endomorphism $\mathfrak{N}$ on $\mathcal{H}$ with the following conditions;
         \begin{align}
         \label{flat nilpotent}
          [\nabla,\mathfrak{N}]=\mathfrak{N}\frac{\mathrm{d}\lambda}{\lambda}, \\
          \label{triviality}
          \mathfrak{N}=p_\lambda^*(\mathfrak{N}|_{\lambda=0}), \\
          \label{pairing nilpotent}
          P(\mathfrak{N}a,b)=P(a,\mathfrak{N}b).
         \end{align}
       We obtain a mixed trTLEP-structure as follows.
       Let $\mathcal{G}$ be the cokernel of $\mathfrak{N}$. 
       By (\ref{triviality}), $\mathcal{G}$ is a vector bundle over $\mathbb{P}^1_\lambda\times M$
       such that $p_{\lambda *}p^*_\lambda\mathcal{G}\to\mathcal{G}$ is an isomorphism. 
       Condition (\ref{flat nilpotent}) implies that $\nabla$ induces a flat connection $\overline{\nabla}$
       on $\mathcal{G}$.
       Let $W=(W_k\mid k\in \mathbb{Z})$ be a filtration on $\mathcal{G}$ defined by 
        \begin{equation}\label{limit weight}
        W_k:=
        \begin{cases} 0 &(k<0)\\
        \mathrm{Im}\big{(}\mathrm{Ker}(\mathfrak{N}^{k+1})\to \mathcal{G}\big{)}& (k\geq 0).      
        \end{cases}
        \end{equation}
       The graded pairing
       $Q=(Q_k:\mathrm{Gr}^W_k(\mathcal{G})\otimes j_\lambda^*\mathrm{Gr}^W_k(\mathcal{G})\to 
       \lambda^{-k}\mathcal{O}_{\mathbb{P}_\lambda^1\times M}\mid k\in\mathbb{Z})$
       is given by
        \begin{equation}\label{limit pairing}
         Q_k([a],[b]):=
         \begin{cases} 0&(k<0)\\
         P(\lambda^{-k}\mathfrak{N}^{k}a,b)&(k\geq 0).
         \end{cases}
        \end{equation}
       Here, $a$ is a local section of $\mathrm{Ker}(\mathfrak{N}^{k+1})$ 
       and $[a]$ is the corresponding class in $\mathrm{Gr}^W_k(\mathcal{G})$. 
       Similarly, $b$ is a local section of $j_\lambda^*\mathrm{Ker}(\mathfrak{N}^{k+1})$ 
       and $[b]$ is the corresponding class in $j_\lambda^*\mathrm{Gr}^W_k(\mathcal{G})$.
       We get the following.
       \begin{lemma}\label{nilpotent mixed}
       The tuple $\mathcal{T}_\mathfrak{N}:=(\mathcal{G},\overline{\nabla},W,Q)$ 
       is a mixed {\rm trTLEP}-structure.
       \end{lemma}
       \begin{proof}       
       By (\ref{triviality}), 
       the adjunction $p_{\lambda *}p^*_\lambda\mathrm{Gr}^W_k\mathcal{G}\to\mathrm{Gr}^W_k\mathcal{G}$
       is an isomorphism for every $k$.
       By (\ref{flat nilpotent}), we have
       $[\nabla,\mathfrak{N}^k]=k\mathfrak{N}^k\mathrm{d}\lambda/\lambda$.
       Therefore, for $a\in \mathrm{Ker}(\mathfrak{N}^{k+1})$, we have 
       $$ \mathfrak{N}^{k+1}\nabla a
         =\nabla\mathfrak{N}^{k+1}a+(k+1)\mathfrak{N}^{k+1}a\frac{\mathrm{d}\lambda}{\lambda}=0.$$
       Hence $\nabla a\in \mathrm{Ker}{\mathfrak{N}^{k+1}}$.
       This implies that the subbundle $W_k$ is $\overline{\nabla}$-flat.
       
       Put $P_k(a,b):=P(\mathfrak{N}^ka,b)$. 
       For a subbundle $\mathcal{J}$ of $\mathcal{H}$, 
       put $\mathcal{J}^{\bot k}:=\{a\in\mathcal{H}\mid P_k(a,b)=0\ \text{for all } b\in j^*_\lambda \mathcal{J}\}$.
       Since $P$ is non-degenerate, we have
       $\mathcal{H}^{\bot k}=\mathrm{Ker}{\mathfrak{N}^K}$ and 
       $\mathrm{Im}\mathfrak{N}^{\bot k}=\mathrm{Ker}\mathfrak{N}^{k+1}$.
       Therefore, we have 
       \begin{align}\label{Pk}
       \big{(}\mathrm{Ker}\mathfrak{N}^{k+1}\big{)}^{\bot k}
       =\mathrm{Im}\mathfrak{N}+\mathrm{Ker}{\mathfrak{N}^k}.
       \end{align}
       For   
       $a\in\mathrm{Ker}(\mathfrak{N}^{k+1})$, and $b\in j^*\mathrm{Ker}(\mathfrak{N}^{k+1})$, 
       Let $[a]\in \mathrm{Gr}^W_k(\mathcal{G}),[b]\in j^*\mathrm{Gr}^W_k(\mathcal{G}))$ 
       be the corresponding classes.
       The relation $\lambda^k Q_k([a],[b])=P_k(a,b)$ and $(\ref{Pk})$
       shows that $Q_k$ is well defined and non-degenerate.
       
       Let $a,b$ and $[a],[b]$ as above. 
       We have
       \begin{align*}
        &\mathrm{d}Q_k([a],[b])-Q_k(\overline{\nabla}[a],[b])-Q_k([a],\overline{\nabla}[b])\\
        &=\frac{1}{\lambda^k}\Big{\{}(-k)P(\mathfrak{N}^ka,b)\frac{\mathrm{d}\lambda}{\lambda}
        +\mathrm{d}P(\mathfrak{N}^ka,b)-P(\mathfrak{N}^k\nabla a,b)-P(\mathfrak{N}^ka,\nabla b)
         \Big{\}}\\
        &=\frac{1}{\lambda^k}
          \Big{\{}
          \big{(}(-k)P(\mathfrak{N}^ka,b)+kP(\mathfrak{N}^ka,b)\big{)}\frac{\mathrm{d}\lambda}{\lambda}
          +\mathrm{d}P(\mathfrak{N}^ka,b)-P(\nabla\mathfrak{N}^k a,b)-P(\mathfrak{N}^ka,\nabla b) 
          \Big{\}}\\
        &=0.      
       \end{align*}
       This proves the flatness of $Q_k$.
       \end{proof}

       \subsubsection{Logarithmic trTLEP-structure and limit mixed trTLEP-structure}
      Let $Z$ be a normal crossing hypersurface of a complex manifold $M$.
      Recall the definition of logarithmic trTLEP-structure.
      \begin{definition}[{\cite[Definition 1.8]{R1}}]\label{logarithmic trTLEP}
      Let $k$ be an integer.
      A {\bf trTLEP$(k)$-structure} on $M$ {\bf logarithmic along $Z$}
      is a tuple $\mathcal{T}=(\mathcal{H},\nabla,P)$ with the following properties.
      \begin{itemize}
       \item $\mathcal{H}$ is a holomorphic vector bundle over $\mathbb{P}^1\times M$ 
             such that the adjoint morphism $p^*p_*\mathcal{H}\to \mathcal{H}$ is an isomorphism.
       \item $\nabla$ is a meromorphic flat connection on $\mathcal{H}$ such that 
             \begin{align}
             \nabla: \mathcal{H} 
                     \to 
                     \mathcal{H} \otimes 
                     \Omega_{\mathbb{P}^1_\lambda \times M}(\log Z_0)
                     \otimes
                     \mathcal{O}_{\mathbb{P}^1_\lambda \times M}(\{0\}\times M)                   
            \end{align}
             where $Z_0:={(}\{0, \infty \} \times M) \cup (\mathbb{P}^1_\lambda \times Z)$.
       \item $P:\mathcal{H}\otimes j^*\mathcal{H}\to \lambda^k\mathcal{O}_{\mathbb{P}^1_\lambda\times M}$
             is a $(-1)^k$-symmetric, non-degenerate, $\nabla$-flat pairing.            
      \end{itemize}
      We also call $\mathcal{T}$ a {\bf logarithmic trTLEP($k$)-structure (or log$Z$-trTLEP($k$)-structure)}
      for short.
      \end{definition}
       We also recall the notion of logarithmic Frobenius type structure. 
       \begin{definition}[{\cite[Definition 1.6]{R1}}]
        Let $K$ be a holomorphic flat bundle on $M$.
        Let $\mathcal{U}$ and $\mathcal{V}$ be endomorphisms on $K$.
        A tuple $(\nabla^{\rm r},\mathcal{C},\mathcal{U},\mathcal{V})$
        is called Frobenius type structure on $K$ with logarithmic pole along $Z$
        if 
        \begin{itemize}
        \item $\nabla^{\rm r}$ is a flat connection on $K$ with logarithmic pole along $Z$,
        \item $\mathcal{C}$ is a Higgs field on $K$ with logarithmic pole along $Z$, 
        \end{itemize}
        and these data satisfy the relations $(\ref{A})$ and $(\ref{B})$.
       We also call the tuple $(\nabla^{\rm r},\mathcal{C},\mathcal{U},\mathcal{V})$ 
       a {\bf logarithmic Frobenius type structure} 
       for short.
       \end{definition}
       \begin{remark}\begin{itemize}
       \item 
        This definition of logarithmic Frobenius type structure lacks the pairing.
       \item
        If we assume that $Z$ is smooth, 
        we have the residue endomorphisms 
        $\mathrm{Res}_{Z}\nabla^{\rm r}$, and $\mathrm{Res}_{Z}\mathcal{C}$. 
        \end{itemize}
       \end{remark}
       The following lemma is proved by the same way as Lemma \ref{Frob type trTLE}. 
      \begin{lemma}[{\cite[Proposition 1.10]{R1}}]
       Let $(\mathcal{H},\nabla,P)$ be a logarithmic {\rm trTLEP}$(0)$-structure.
       Then there is a unique logarithmic Frobenius type structure 
       $(\nabla^r,\mathcal{C},\mathcal{U},\mathcal{V})$ on $\mathcal{H}|_{\lambda=0}$
       such that 
       \begin{align}
        \nabla=p_\lambda^*\nabla^{\rm r}+
                \frac{1}{\lambda}p_\lambda^*\mathcal{C}
                +\Big{(}\frac{1}{\lambda}p_\lambda^*\mathcal{U}-p_\lambda^*\mathcal{V}\Big{)}
                \frac{\mathrm{d}\lambda}{\lambda}.
       \end{align}\qed
      \end{lemma}
      
      In the following, we assume that $Z$ is smooth. 
      Let $\mathcal{T}=(\mathcal{H},\nabla,P)$ be a trTLEP(0)-structure on $M$ logarithmic along $Z$ such that 
      \begin{equation}\label{limit res}
        \mathrm{Res}_{\mathbb{P}^1_\lambda\times Z}(\nabla)\mid_{\{\infty\}\times Z}=0.
      \end{equation}
      Fix a point $z$ in $Z$ and a defining function $q$ of $Z$ on a neighborhood of $z$.
      Then the residual connection $\nabla^q$ on $\mathcal{H}\mid_{\mathbb{P}^1_\lambda\times(Z,z)}$ is induced.  
      It is easy to see that the tuple 
      $\mathcal{T}^q:=(\mathcal{H}\mid_{\mathbb{P}^1_\lambda\times(Z,z)},\nabla^q,P\mid_{(Z,z)})$
      is a trTLEP(0)-structure on the germ $(Z,z)$ of a complex manifold.
      
       \begin{lemma}The endomorphism
         $\mathfrak{N}:=\lambda\mathrm{Res}_{\mathbb{P}^1_\lambda\times Z}(\nabla)$ is nilpotent and 
         satisfies the conditions 
         $(\ref{flat nilpotent})$--$(\ref{pairing nilpotent})$
         with respect to the {\rm trTLEP}$(0)$-structure
         $\mathcal{T}^q=(\mathcal{H}\mid_{(Z,z)},\nabla^q,P\mid_{(Z,z)})$.
       \end{lemma}
       \begin{proof}
       First of all, we show that $\mathrm{Res}_{\mathbb{P}^1_\lambda\times Z}(\nabla)$ is nilpotent.
       Let $(\nabla^{\rm r},\mathcal{C},\mathcal{U},\mathcal{V})$ be the 
       logarithmic Frobenius type structure on $\mathcal{H}|_{\lambda=0}$
       such that 
       \begin{align}
        \nabla=p_\lambda^*\nabla^{\rm r}+
                \frac{1}{\lambda}p_\lambda^*\mathcal{C}
                +\Big{(}\frac{1}{\lambda}p_\lambda^*\mathcal{U}-p_\lambda^*\mathcal{V}\Big{)}
                \frac{\mathrm{d}\lambda}{\lambda}.
       \end{align}
       The condition (\ref{limit res}) is equivalent to $\mathrm{Res}_{Z,z}(\nabla^r)=0$.
       Then we have 
       $\mathrm{Res}_{\mathbb{P}^1_\lambda\times(Z,z)}(\nabla)=\lambda^{-1}\mathrm{Res}_{Z,z}\mathcal{C}$.
       Since the eigenvalues of $\mathrm{Res}_{\mathbb{P}^1_\lambda\times(Z,z)}(\nabla)$
       and $\mathrm{Res}_{Z,z}\mathcal{C}$ 
       are both constant along $\mathbb{P}^1_\lambda$, 
       they are all zero. 
       Therefore, the endomorphism $\mathrm{Res}_{\mathbb{P}^1_\lambda\times(Z,z)}(\nabla)$ 
       is nilpotent.
       
       The following shows (\ref{flat nilpotent}):
       \begin{align*}
        [\nabla^q,\mathfrak{N}]
                               &=\lambda[\nabla^q,\mathrm{Res}_{\mathbb{P}^1_\lambda\times(Z,z)}(\nabla)]
                                 +\lambda\mathrm{Res}_{\mathbb{P}^1_\lambda\times(Z,z)}(\nabla)
                                  \frac{\mathrm{d}\lambda}{\lambda}\\
                               &=\mathfrak{N}\frac{\mathrm{d}\lambda}{\lambda}.
       \end{align*}
       The condition (\ref{triviality}) is clear by $\mathfrak{N}=\mathrm{Res}_{Z,z}\mathcal{C}$.
       The flatness of $P$ implies (\ref{pairing nilpotent}).
       \end{proof}
       This lemma together with Lemma \ref{nilpotent mixed} defines mixed trTLEP-structure on $(Z,z)$.
       \begin{definition}\label{LMtrTLEP}
         We call the mixed trTLEP-structure 
         $\mathcal{T}_{Z,z}:=\mathcal{T}^q_\mathfrak{N}$ 
         a {\bf limit mixed trTLEP structure}.
       \end{definition}
       \begin{remark}
         The mixed trTLEP-structure $\mathcal{T}_{Z,z}$ does not depend on the choice of 
         the defining function $q$ of $Z$ around $z$.
       \end{remark}

      \subsection{Quantum D-modules on A-models}\label{QDM A}
       We recall the definition of quantum D-modules and their properties. 
       Let $X$ be a smooth projective toric variety and $\Lambda_X\subset H_2(X,\mathbb{C})$ a 
       semi-subgroup consists of effective classes.
       For each $d\in \Lambda_X$ and $n\in \mathbb{Z}_{\geq 0}$, 
       let $\overline{X}_{0,n,d}$ be the moduli space of genus 0 stable maps to $X$ of degree $d$.
       We denote the $i$-th evaluation map by $ev_i:\overline{X}_{0,n,d}\to X$.
       Fix a Hermitian metric $\|*\|$ on $H^{\neq 2}(X,\mathbb{C})$
       \begin{theorem}[{\cite[Theorem 1.3]{Iri}, See also \cite[Theorem 4.2]{RS}}]\label{domain}
        The Gromov-Witten potential 
        \begin{align}
             \Phi_X(\tau)
              &=\sum_{d \in \Lambda_X}\sum_{n\geq 0}\frac{1}{n!}
                \int_{[\overline{X}_{0,n,d}]^{\rm vir}}\prod_{i=1}^nev^*_i(\tau) \\
              &=\sum_{d \in \Lambda_X}\sum_{n\geq 0}\frac{1}{n!}
                 \int_{[\overline{X}_{0,n,d}]^{\rm vir}}ev^*_i(\tau')e^{\delta(d)}.
          \end{align} 
        converges on a simply connected domain 
         \begin{equation}
           U_X:=\{\tau=\tau'+\delta \in H^*(X,\mathbb{C})|\ \mathrm{Re}(\delta(d))<-C,\ \|\tau'\|<e^{-C}\}
        \end{equation}  
        for $C\gg 0$. Hence $\Phi_X(\tau)$ defines a holomorphic function on $U_X$.\qed
        \end{theorem}
        
       Define $g:H^*(X,\mathbb{C})\otimes H^*(X,\mathbb{C})\to \mathbb{C}$ by 
       $g(\alpha,\beta):=\int \alpha \cup \beta$, 
       and a quantum cup product $\circ$ on $H^*(X,\mathbb{C})$ by 
       \begin{equation}
        g(\alpha\circ_\tau \beta,\gamma)=\alpha \beta \gamma \Phi_X(\tau)\ 
        (\alpha,\beta,\gamma\in H^*(X,\mathbb{C})).
       \end{equation}
       Here, $\tau\in U_X$ and we regard $\alpha,\beta,\gamma$ on the right hand side as differential operators.
       Let $E$ be a vector field on $H^*(X,\mathbb{C})$ defined as a sum of first Chern class $c_1(X)$ 
       and fundamental vector field of the action of $\mathbb{C}^*$ defined by 
       $t\cdot \alpha=t^{\frac{2-i}{2}}\alpha\ (\alpha \in H^i(X,\mathbb{C}))$. 
       Then it is well known that
         $\mathscr{F}_X:=(\circ,E,g)$ is a Frobenius structure 
         on $U_X$ 
         of charge $\dim X$.
       \begin{definition}[{\cite[Definition-Lemma 4.3]{RS}}]
         We call the {\rm trTLEP$(0)$}-structure $\mathcal{T}(\mathscr{F}_X)$ on $U_X$ a {\bf big quantum D-module}.
         We also call its restriction to $U'_X:=U_X\cap H^2(X,\mathbb{C})$ a {\bf small quantum D-module}.
       \end{definition}
       \begin{remark}
        This definition is equivalent to the definition in {\rm \cite{RS}}.
        The {\rm trTLEP}$(\dim X)$-structure $\mathcal{T}(\mathscr{F}_X)(\dim X/2)$ is considered there 
        $($see Remark $\ref{remark on MtrTLEP}$$)$.
       \end{remark}
       Denote by $V_{X}^0$ (resp. $V'^0_{X}$) 
       the quotient space of $U_X$ (resp. $U'_X$) 
       by the natural action of $2\pi\sqrt{-1}H^2(X,\mathbb{Z})$.
       The following lemma is trivial by construction.
       \begin{lemma}
        The big quantum D-module $\mathcal{T}(\mathscr{F}_X)$ induces 
        a {\rm trTLEP$(0)$}-structure $\mathcal{T}^{\rm big}_X$ 
        on $V_X^0$.
        The small quantum D-module also induces
        a {\rm trTLEP$(0)$}-structure $\mathcal{T}^{\rm small}_X$
        on $V_X^{'0}$.
        \qed
       \end{lemma}
       
       Fix a nef basis $T_1, T_2, \dots T_r$ of $H^2(X,\mathbb{Z})$. 
       Then the embedding of $V'_X$ into $\mathbb{C}^r$ is naturally defined.
       Let $q=(q_1, q_2, \dots, q_r)$ be the canonical coordinate on $\mathbb{C}^r$ and 
       $\mid *\mid$ the canonical Hermitian metric.
       We define 
       \begin{align}\label{big domain}
        V_X&:=\{(q,\tau')\in \mathbb{C}^r\times H^{\neq2}(X,\mathbb{C}) 
        \big{|}|q|<e^{-C},\|\tau'\|<e^{-C}\} \\
        \label{small domain}
        V'_X&:=\{q\in \mathbb{C}^r \big{|}|q|<e^{-C}\}.
       \end{align}
       Then we have the following. 
       \begin{proposition}[{\cite[Corollary 4.5]{RS}}]\label{log on V}
         $\mathcal{T}^{\rm big}_X$ $($resp. $\mathcal{T}^{\rm small}_X$$)$ 
         is extended to a logarithmic {\rm trTLEP}$(0)$-structure on 
         $V_X$ $($resp. $V'_X$$)$.\qed
       \end{proposition}
       \subsection{MFS for local A-models}\label{QDM locA}
        \subsubsection{Construction of MFS for local A-models}
         Let $S$ be a weak Fano toric surface, 
         $\gamma_0=1\in H^0(S,\mathbb{Z})$ a unit, 
         and $\gamma_{r+1}$ its Poincare dual. 
         Fix a nef basis $\gamma_1,\gamma_2,\dots,\gamma_r\in H^2(S,\mathbb{Z})$.
         Then $\{\gamma_0,\gamma_,\dots,\gamma_{r+1}\}$ is a basis of $H^*(S,\mathbb{C})$.
         Denote by $\Lambda_S$ the semi-subgroup consists of effective classes in $H_2(S,\mathbb{Z})$.
         
         Let $X$ be the projective compactification of the 
         canonical bundle $K_S$ (i.e. $X:=\mathbb{P}(K_S\oplus \mathcal{O}_S)$).
         Let $p:X\to S$ be the natural projection and 
         $i:S\to X$ the embedding defined by the zero section of $K_S$.
         Put $\Gamma_i:=p^*\gamma_i\ (0\leq i\leq r+1)$ and 
         $\Delta_0:=c_1(\mathcal{O}_{X/S}(1)) \in H^2(X,\mathbb{Z}) $.
         Put $\Delta_i:=\Delta_0\cup\Gamma_i\ (0\leq i\leq r+1)$.
         Then the classes $\Gamma_i, \Delta_j\ (0\leq i,j \leq r+1)$ form a basis of $H^*(X,\mathbb{C})$.
         This basis gives a coordinate 
         $$(t,s)=(t^0,\dots, t^i,\dots,t^{r+1},s^0,\dots, s^j,\dots s^{r+1})$$
         on $H^*(X,\mathbb{C})$. Put $q_0:=e^{s_0}$ and $q_i:=e^{t_i}$.
         Then, the quantum cup product on $H^*(X,\mathbb{C})$ 
         is given as follows.         
         \begin{lemma}[{\cite[Lemma 8.9, 8.10]{konisi1}}]\label{local cup}
            \begin{align}\label{Delta}
            &\Delta_i\circ*=\Delta_i\cup* + O(q_0), \\
            \label{toric GW prod}
            &\Gamma_i\circ \Gamma_j=\Gamma_i\cup \Gamma_j+\sum_{k=1}^r\sum_{d\in i_*{\Lambda_S}\setminus \{0\}}
            (\Gamma_i(d)\Gamma_j(d)\Gamma_k(d)N_dq^d)\Gamma^\vee_k+O(q_0).
            \end{align}
            Here, $q^d:=\prod_{i=1}^rq^{\langle d,\Gamma_i \rangle}$ and 
            $N_d:=\int_{[\overline{X}_{0,0,d}]^{\rm vir}}1$.
            The symbol $O(q_0)$ represents the higher order term with respect to $q_0$. \qed
         \end{lemma}
         By Proposition \ref{log on V}, 
         we have logarithmic trTLEP$(0)$-structure 
         $\mathcal{T}^{\rm small}_X$
         on $V'_X$. (Note that $q=(q_0,q_1,\dots,q_r)\in\mathbb{C}^{r+1}$.) 
         Put $M:=V'_X\setminus\bigcup_{i>0}\{q_i=0\}$ and 
         denote the restriction of $\mathcal{T}^{\rm small}_X$
         to $M$ by the same latter.
         Let $Z$ be a divisor of $M$ defined by $q_0=0$.
         For each point $z$ in $Z$, 
         the restriction of $\mathcal{T}^{\rm small}_X$
         to the germ $(M,z)$ satisfies the condition (\ref{limit res})
         along $(Z,z)$. 
         Hence we get the limit mixed trTLEP-structure $(\mathcal{T}^{\rm small}_X)_{Z,z}$.
         To compare with the results of Konishi-Minabe \cite{konisi1} later, 
         we consider the Tate twist $(\mathcal{T}^{\rm small}_X)_{Z,z}(-1/2)$ (See Remark \ref{remark on MtrTLEP}).           
         \begin{proposition}\label{Frobenius type for limit}
         Let $(\mathcal{T}^{\rm small}_X)_{Z,z}(-1/2)=(\mathcal{H},W,P)$ be the mixed trTLEP-structure on $(Z,z)$
         constructed above 
         and $(\nabla^{\rm r},\mathcal{C},\mathcal{U},\mathcal{V})$ the corresponding Frobenius type structure.
         $($See Lemma {\rm \ref{Frob type trTLE}}.$)$
         Put $\mathfrak{N}:=\Delta \cup *:H^*(X)\to H^*(X)$.
         Then we get the following.
              \begin{align}
              \label{co-kernel}
              \mathcal{H}|_{\lambda=0}&=\mathrm{Cok} (\mathfrak{N})\times (Z,z) \simeq 
              \Bigg(\bigoplus_{i=0}^{r+1}\mathbb{C}\Gamma_i\Bigg)\times (Z,z),\\
              \label{surf filter}
              W_k|_{\lambda=0}&=
              \begin{cases}0\ &(k\leq 0)\\
              \mathrm{Im}(\mathrm{Ker}(\mathfrak{N}^{k})\to \mathcal{H}|_{\lambda=0}) &(k>0),              
              \end{cases} 
              \\ \label{limit cup}
              \mathcal{C}_{q_i\partial_{q_i}}(\Gamma_j)&=
              \begin{cases}
              \Gamma_i,&(j=0)\\
              \Gamma_i\cup \Gamma_j+\sum_{k=1}^r
              \sum_{d\in i_*{\Lambda_S} \setminus \{0\}}
              (\Gamma_i(d)\Gamma_j(d)\Gamma_k(d)N_dq^d)\Gamma^\vee_k
              & (j>0),
              \end{cases}
              \\
              \mathcal{U}&=0,\\
              \label{mathcal V}
              \mathcal{V}&=-\frac{\deg}{2}+2,
              \end{align}
              where $i=1,2,\dots,r$, 
              $q^d$ and $N_d$ in the equation $(\ref{limit cup})$ are defined as in Lemma $\ref{local cup}$, 
              and the operator $\deg$ is defined by $\deg(\Gamma_j):=m\Gamma_j$ for $\Gamma_j\in H^m(X)$.
         \end{proposition}
         \begin{proof}
         The Lemma $\ref{local cup}$ implies (\ref{co-kernel}) and (\ref{limit cup}).
         In particular, we have $\lambda\mathrm{Res}_{\mathbb{P}^1_\lambda\times(Z,z)}\nabla=\Delta\cup*$ 
         for the connection $\nabla$ underlying $\mathcal{T}^{\rm small}_X$.
         The equation \ref{limit weight} twisted by $(-1/2)$ gives (\ref{surf filter}).
         As is shown in \cite{konisi1}, the Euler vector field $E$ of $\mathscr{F}_X$ is given by 
         $$E=t^0\frac{\partial}{\partial t^0}+2\frac{\partial}{\partial s^0}
         -t^{r+1}\frac{\partial}{\partial t^{r+1}}
         -\sum_{i=1}^r s^i\frac{\partial}{\partial s^i}-2s^{r+1}\frac{\partial}{\partial s^{r+1}}.$$
         Hence $\mathcal{C}_E$ is $0$ on $\mathrm{Cok}(\mathfrak{N})$ over $Z$. 
         Since the charge of $\mathscr{F}_X$ is $3$, considering the twist, we have 
         $\mathcal{V}=\overline{\nabla E}-(2-3)/2+1/2=-\deg/2+2$, 
         where $\overline{\nabla E}$ is the endomorphism induced on $\mathrm{Cok}(\mathfrak{N})$ over $Z$ 
         by $\nabla E$.
         \end{proof} 
         Using this proposition,
         we get the following.
         \begin{theorem}\label{loc A thm}
         If $z\in Z$ is in a sufficiently small neighborhood of $0\in V'_X$, 
         there exits a tuple 
         $\big{(}(\widetilde{M},0), \mathscr{F}_A^{\rm loc},\iota,i\big{)}$ 
         with the following conditions 
         uniquely up to isomorphisms.
          \begin{enumerate}
          \item[$1$.] $\mathscr{F}^{\rm loc}_A$ 
          is a MFS of charge $4$ on a germ  of complex manifold $(\widetilde{M},0)$.
          \item[$2.$] $\iota:(Z,z)\hookrightarrow (\widetilde{M},0)$ is a closed embedding.
          \item[$3.$] $i:(\mathcal{T}^{\rm small}_X)_{Z,z}(-1/2)\xrightarrow{\sim}\iota^*\mathcal{T}(\mathscr{F})$
          is an isomorphism of mixed trTLEP-structure such that the restriction $i|_{(0,z)}$
          sends $\Gamma_0$ to the unit vector field of $\mathscr{F}$.
          \end{enumerate}
         \end{theorem}
         \begin{proof}
         By (\ref{surf filter}), (\ref{limit cup}), and (\ref{mathcal V}),
         $\Gamma_0$ satisfies (IC), (GC), and $(EC)_4$ when $z$ is sufficiently small. 
         Therefore, by Corollary \ref{construction theorem of mixed Frobenius manifold}, 
         we have the conclusion.
         \end{proof}
         
         \subsubsection{Comparison with the result of Konishi and Minabe}
          The mixed Frobenius manifold $\mathscr{F}^{\rm loc}_A$ constructed in Theorem \ref{loc A thm} 
          is isomorphic to the mixed Frobenius manifold 
          constructed in \cite{konisi1} as follows. 
          Let $\mathscr{F}_{\rm KM}$ be the mixed Frobenius structure on a open subset of $H^*(S,\mathbb{C})$ 
          defined in \cite[Theorem 8.7]{konisi1}.
          Regard $Z$ as a subset of the quotient 
          $H^2(S,\mathbb{C})/2\pi\sqrt{-1} H^2(S,\mathbb{Z})$ 
          via the pull back $p^*$.
          It is easy to see that $\mathscr{F}_{\rm KM}$ induces MFS on 
          $H^*(S,\mathbb{C})/2\pi\sqrt{-1}H^*(S,\mathbb{Z})$
          , which we denote by the same notation.
          We restrict the induced mixed trTLEP-structure $\mathcal{T}(\mathscr{F}_{\rm KM})$
          to the germ $(Z,z)$ and denote it by $\mathcal{T}(\mathscr{F}_{\rm KM})|_{(Z,z)}$.
          We have the following proposition.
          \begin{proposition}
            We have a natural isomorphism 
            $\mathcal{T}(\mathscr{F}_{\rm KM})|_{(Z,z)}\simeq (\mathcal{T}^{\rm small}_X)_{Z,z}(-1/2)$.
          \end{proposition}
          \begin{proof}
          The isomorphism is given by $\gamma_i\mapsto \Gamma_i$ $(0\leq i\leq r+1)$.
          Comparing the proof of \cite[Theorem 8.7]{konisi1} 
          with (\ref{limit pairing}) and  Proposition \ref{Frobenius type for limit}, 
          we can check that this gives an isomorphism of mixed trTLEP-structure over $(Z,z)$.
          \end{proof}
          This proposition together with the uniqueness in Theorem \ref{loc A thm} shows the following.
          \begin{corollary}
           We have an isomorphism of mixed Frobenius manifolds
           $$\big{(}(\widetilde{M},0),\mathscr{F}^{\rm loc}_A\big{)}\simeq 
           \big{(}(H^*(S,\mathbb{C})/2\pi\sqrt{-1}H^*(S,\mathbb{Z}),z),\mathscr{F}_{\rm KM}\big{)}.$$
           \qed
          \end{corollary}
          This shows that we have constructed the mixed Frobenius manifold $\mathscr{F}_{\rm KM}$
          by using the limit mixed trTLEP-structure and the unfolding theorem.

%


\end{document}